\documentclass[11pt]{article}

\usepackage{amssymb}
\usepackage{amsmath,amsfonts,mathtools}
\usepackage{amsthm}
\usepackage{latexsym}
\usepackage{mathrsfs}
\usepackage{color}
\usepackage{float}
\usepackage{enumitem}
\usepackage{algorithm,algorithmic}
 \allowdisplaybreaks[3]
 
 \usepackage{bbm}
\usepackage{microtype}
\usepackage[a4paper]{geometry}
\usepackage{pifont}
\usepackage{latexsym}

\usepackage[a4paper]{geometry}
\newcommand{\beq}{\begin{equation}}
\newcommand{\eeq}{\end{equation}}
\newcommand{\beqa}{\begin{eqnarray}}
\newcommand{\eeqa}{\end{eqnarray}}
\newcommand{\beqas}{\begin{eqnarray*}}
\newcommand{\eeqas}{\end{eqnarray*}}
\newcommand{\ba}{\begin{array}}
\newcommand{\ea}{\end{array}}
\newcommand{\bi}{\begin{itemize}}
\newcommand{\ei}{\end{itemize}}

\newcommand{\mcK}{{\mathcal K}}

\newcommand{\mcL}{{\mathcal L}}

\newcommand{\mcN}{{\mathcal N}}
\newcommand{\mcJ}{{\mathcal J}}
\newcommand{\mcT}{{\mathcal T}}
\newcommand{\mcI}{{\mathcal I}}
\newcommand{\cS}{{\mathcal S}}

\newcommand{\mcQ}{{\mathcal Q}}
\newcommand{\chS}{{\widehat{\mathcal{S}}}}

\newcommand{\prox}{\mathrm{prox}}
\newcommand{\dom}{\mathrm{dom}}
\newcommand{\dist}{\mathrm{dist}}
\newcommand{\argmin}{\arg\min}

\newtheorem{lemma}{Lemma}
\newtheorem{thm}{Theorem}

\newtheorem{defi}{Definition}

\newtheorem{ass}{Assumption}

\newtheorem{rem}{Remark}

\newcounter{spb}
\def\br{{\bar r}}

\def\cC{{\Omega}}
\def\cl{{\rm cl}}
\def\cN{{\cal N}}
\def\cO{{\cal O}}
\def\cU{{\cal U}}

\def\tv{{\tilde v}}
\def\tx{{\tilde x}}
\def\tg{{\tilde \gamma}}
\def\tn{{\tilde n}}
\def\wC{{\widehat C}}

\def\tL{{\widehat L}}
\def\tB{{\widehat B}}
\def\tC{{\widetilde C}}

\def\tU{{\widetilde U}}
\def\tr0{{\tilde r_0}}

\def\rr{{\mathbb{R}}}
\def\bbK{{\mathbb{K}}}
\def\bbZ{{\mathbb{Z}}}
\def\ctQ{{\widetilde \mcQ}}
\def\chQ{{\widehat \mcQ}}

\title{Accelerated first-order methods for convex optimization with locally Lipschitz continuous gradient}

\author{
Zhaosong Lu
\thanks{
Department of Industrial and Systems Engineering, University of Minnesota, USA (email: {\tt zhaosong@umn.edu}, {\tt mei00035@umn.edu}).}
\and
Sanyou Mei
\footnotemark[1]
}

\date{June 1, 2022 (Revised: April 8, 2023)}

\begin{document}
\maketitle

\begin{abstract}
In this paper we develop accelerated first-order methods for convex optimization with \emph{locally Lipschitz} continuous gradient (LLCG), which is beyond the well-studied class of convex optimization with  
Lipschitz continuous gradient. In particular, we first consider unconstrained convex optimization with LLCG and propose accelerated proximal gradient (APG) methods for solving it. The proposed APG methods are equipped with a verifiable termination criterion and enjoy an operation complexity of $\cO(\varepsilon^{-1/2}\log \varepsilon^{-1})$ and $\cO(\log \varepsilon^{-1})$ for finding an $\varepsilon$-residual solution of an unconstrained convex and strongly convex optimization problem, respectively. We then consider constrained convex optimization with LLCG and propose an first-order proximal augmented Lagrangian method for solving it by applying one of our proposed APG methods to approximately solve a sequence of proximal augmented Lagrangian subproblems. The resulting method is equipped with a verifiable termination criterion and enjoys an operation complexity of $\cO(\varepsilon^{-1}\log \varepsilon^{-1})$ and $\cO(\varepsilon^{-1/2}\log \varepsilon^{-1})$ for finding an $\varepsilon$-KKT solution of a constrained convex and strongly convex optimization problem, respectively. All the proposed methods in this paper are \emph{parameter-free} or \emph{almost parameter-free} except that the knowledge on convexity parameter is required. In addition, preliminary numerical results are presented to demonstrate the performance of our proposed methods. To the best of our knowledge, no prior studies were conducted to investigate accelerated first-order methods with complexity guarantees for convex optimization with LLCG.   All the complexity results obtained in this paper are new.
\end{abstract}


\medskip


\section{Introduction}

In this paper we first consider unconstrained convex optimization\footnote{We refer to problem \eqref{unc-prob} as an unconstrained optimization problem just for convenience. Strictly speaking, it can be a constrained optimization problem. For example, when $P$ is the indicator function of a closed convex set, it reduces to the problem of minimizing $f$ over this set.}
\begin{equation}\label{unc-prob}
F^*=\min_{x} \ \{F(x)\coloneqq f(x)+ P(x)\}, 
\end{equation}
where $F^*\in \rr$, $f,  P:\rr^n\to(-\infty,\infty]$ are proper closed convex functions, $f$ is differentiable on $\cl(\dom(P))$, and $\nabla f$ is \emph{locally Lipschitz} continuous\footnote{See Subsection \ref{notation} for the definition of locally Lipschitz continuity.} on $\cl(\dom(P))$, where 
$\dom(P)$ denotes the domain of $P$ and $\cl(\dom(P))$ denotes 
its closure.  It shall be mentioned that $\dom(P)$ is possibly \emph{unbounded}.
Problem \eqref{unc-prob} is beyond the well-studied class of problems in the form of \eqref{unc-prob} yet with $\nabla f$ being (\emph{globally}) \emph{Lipschitz} continuous on $\cl(\dom(P))$ or $\rr^n$.  For example, the problem of minimizing a convex high-degree polynomial function over a closed unbounded convex set is a special case of \eqref{unc-prob}, but it does not belong to the latter class in general. In addition, it is sometimes easier to verify 
local Lipschitz continuity than Lipschitz continuity of $\nabla f$ on $\cl(\dom(P))$. For example, when $f$ is twice differentiable in an open set containing $\cl(\dom(P))$, it is straightforward to see that 
$\nabla f$ is locally Lipschitz continuous on $\cl(\dom(P))$; however, verifying Lipschitz continuity of $\nabla f$ may require exploring the expression of $\nabla f$ and can be a nontrivial task.   

The well-known special case of problem \eqref{unc-prob} with $\nabla f$ being 
 \emph{Lipschitz} continuous on $\cl(\dom(P))$ or $\rr^n$ has been extensively studied in the literature. In particular, accelerated proximal gradient (APG) methods \cite{BecTeb09,Nest13} and their variants \cite{FeRiTa15,LiLuXi14,Tse08} were proposed for solving it. From theoretical perspective, these methods enjoy an optimal iteration complexity of $\cO(\varepsilon^{-1/2})$  for finding an $\varepsilon$-gap solution of \eqref{unc-prob}, namely, a point $x$ satisfying $F(x)-F^* \leq\varepsilon$. However, since $F^*$ is typically unknown, there is a lack of a verifiable termination criterion for them to find an $\varepsilon$-gap solution of \eqref{unc-prob} in general. To overcome this issue, a nearly optimal proximal gradient method was recently proposed in \cite{ito2021nearly} for solving such a special case of \eqref{unc-prob}. This method is equipped with a verifiable termination criterion based on the norm of a gradient mapping of \eqref{unc-prob} and enjoys an iteration complexity of $\cO(\varepsilon^{-1/2}\log \varepsilon^{-1})$ for finding an $\varepsilon$-norm solution of \eqref{unc-prob}, namely, a point  at which the norm of a gradient mapping of \eqref{unc-prob} is no more than $\varepsilon$. It shall be mentioned that these methods \cite{BecTeb09,FeRiTa15,ito2021nearly,LiLuXi14,Nest13,Tse08} and their analysis rely on the \emph{Lipschitz} continuity of $\nabla f$ on $\cl(\dom(P))$ or $\rr^n$.  
Indeed, they require either an explicitly known global Lipschitz constant of $\nabla f$ \cite{FeRiTa15,LiLuXi14,Tse08} or an estimated one obtained by a backtracking line search scheme \cite{BecTeb09,ito2021nearly,Nest13}. When $\nabla f$ is merely locally Lipschitz continuous, a global Lipschitz constant of $\nabla f$ clearly does not exist and also the sequence of estimated Lipschitz constants in \cite{BecTeb09,ito2021nearly,Nest13} can blow up because the solution sequence is possibly unbounded. If the latter case occurs, the methods may not converge and the complexity analysis of the methods in \cite{BecTeb09,ito2021nearly,Nest13} will no longer hold. As a result, these methods are not applicable to \eqref{unc-prob} or lack complexity guarantees in general when $\nabla f$ is merely \emph{locally Lipschitz} continuous on $\cl(\dom(P))$. 

To handle the challenge of the local Lipschitz continuity of $\nabla f$, we modify \cite[Algorithm 1 with a single block]{LiLuXi14} by incorporating a backtracking line search scheme and an adaptive update strategy on the algorithm parameters to propose an APG method (see Algorithm~\ref{alg-acc}) for solving problem \eqref{unc-prob}. Interestingly, the solution sequence and the sequence of estimated  (local) Lipschitz constants obtained by the proposed APG method can be proved to be bounded, which overcome the aforementioned issues of the methods in \cite{BecTeb09,ito2021nearly,Nest13}. 
Moreover, this method is shown to enjoy a nice iteration complexity of $\cO(\varepsilon^{-1/2})$ and $\cO(\log\varepsilon^{-1})$ for finding an $\varepsilon$-gap solution of \eqref{unc-prob} when $f$ is convex and strongly convex, respectively. Yet, since $F^*$ is typically unknown, it is difficult to come up with a verifiable termination criterion for this method to find an $\varepsilon$-gap solution of \eqref{unc-prob}. To circumvent this issue, we further propose an APG method with a \emph{verifiable} termination criterion (see Algorithm~\ref{alg-acc-term}) for \eqref{unc-prob} with a \emph{strongly convex} $f$, and show that it enjoys an iteration and operation complexity\footnote{The operation complexity of a proximal gradient method for problem \eqref{unc-prob} is measured by the amount of its fundamental operations consisting of evaluations of $\nabla f$ and proximal operator of $P$.}  of $\cO(\log \varepsilon^{-1})$ for finding an $\varepsilon$-residual solution of \eqref{unc-prob}, namely, a point $x$ satisfying $\dist(0,\partial F(x))\leq\varepsilon$.\footnote{$\dist(z,\Omega)=\min_y\{\|z-y\|: y\in\Omega\}$ for any $z\in\rr^n$ and closed set $\Omega\subseteq \rr^n$. In addition, an $\varepsilon$-residual solution $x$ of \eqref{unc-prob} satisfying $\|x\|\leq \Delta$ for some $\Delta>0$ independent on $\varepsilon$ is an $\cO(\varepsilon)$-gap solution, because $F(x)-F^* \leq \|x-x^*\| \dist(0,\partial F(x)) \leq (\Delta+\|x^*\|)\varepsilon$ for any optimal solution $x^*$ of \eqref{unc-prob}. However, the converse may not be true.} We also propose an APG method with a \emph{verifiable} termination criterion (see Algorithm \ref{PPA-sp}) for \eqref{unc-prob} with a \emph{convex but non-strongly convex} $f$ by applying Algorithm~\ref{alg-acc-term} to a sequence of strongly convex optimization problems arising from a perturbation of \eqref{unc-prob}, and show that it enjoys an operation complexity of $\cO(\varepsilon^{-1/2}\log \varepsilon^{-1})$ for finding an $\varepsilon$-residual solution of \eqref{unc-prob}. All the proposed APG methods are \emph{parameter-free} or \emph{almost parameter-free} except that the knowledge on convexity parameter of $f$ is required.

Secondly, we consider constrained convex optimization in the form of
\begin{equation}\label{conic-p}
\begin{array}{rl}
\bar F^*=\min & \left\{F(x) := f(x) +  P(x)\right\} \\
\mbox{s.t.} &  -g(x)\in \mcK,
\end{array}
\end{equation}
where $\mcK\subseteq \rr^m$ is a closed convex cone, $f,P:\rr^n\to(-\infty,\infty]$ are proper closed convex functions, $f$ and $g$ are differentiable on $\cl(\dom(P))$, $\nabla f$ and $\nabla g$ are \emph{locally Lipschitz} continuous on $\cl(\dom(P))$, and $g$ is $\mcK$-convex, that is, 
\begin{equation*}
 \alpha g(x) + (1-\alpha)g(y) - g(\alpha x + (1-\alpha)y) \in \mcK, \quad\forall x,y\in\rr^n, \; \alpha\in[0,1].
\end{equation*}
 It shall be mentioned that $\dom(P)$ is possibly \emph{unbounded}. 
 
Problem \eqref{conic-p} includes a rich class of problems as a special case. For example, when $\mathcal{K} = \rr_{+}^{m_1}\times\{0\}^{m_2}$ for some $m_1$ and $m_2$, $g(x) = (g_1(x),\ldots,g_{m_1}(x), h_1(x),\ldots,h_{m_2}(x))^T$ with convex $g_i$'s and affine $h_j$'s, and $P(x)$ is the indicator function of a simple convex set $X\subseteq\rr^n$, problem \eqref{conic-p} reduces to an ordinary convex optimization problem
\[
\begin{array}{rl}
\min\limits_{x \in X} \{f(x): g_i(x) \leq 0, \ i=1,\ldots,m_1; h_j(x) = 0, \ j=1,\ldots,m_2\}.
\end{array}
\]

Numerous first-order methods were developed for solving some special cases of \eqref{conic-p} in the literature. For example, a variant of Tseng's modified forward-backward splitting method was proposed in \cite{monteiro2011complexity} for \eqref{conic-p} with $g$ being an affine map, $\mcK=\{0\}^m$, and $\nabla f$ being \emph{Lipschitz} continuous on $\cl(\dom(P))$. Also, first-order penalty methods were proposed in \cite{LaMo13} for \eqref{conic-p} with $g$ being an affine map, $P$ being the indicator function of a simple \emph{compact} convex set, and $\nabla f$ being \emph{Lipschitz} continuous on this set. In addition, first-order augmented Lagrangian (AL) methods were developed in \cite{AyGa13,NePaGl19} for \eqref{conic-p} with $g$ being an affine map, $P$ having a \emph{bounded} domain or being the indicator function of a simple \emph{compact} convex set, and $\nabla f$ being \emph{Lipschitz} continuous on $\rr^n$. Also, first-order AL methods were proposed in \cite{Lan16,LLM17,PNT17} with $\mcK = \{0\}^m$, $g$ being an affine map, $P$ having \emph{bounded} domain or being the indicator function of a simple \emph{compact} convex set, and $\nabla f$ being \emph{Lipschitz} continuous on this set or $\rr^n$. For these special cases,  first-order iteration complexity was established for the methods  \cite{AyGa13,NePaGl19,PNT17} for finding an $\varepsilon$-gap solution\footnote{An $\varepsilon$-gap solution of 
problem \eqref{conic-p} is a point $x$ satisfying $|F(x) - \bar F^*| \leq \varepsilon$ and $\dist(g(x),-\mcK) \leq \varepsilon$.} of \eqref{conic-p} and for the methods  \cite{LaMo13,Lan16,monteiro2011complexity} for finding an $\varepsilon$-KKT type solution, which is similar to the one introduced in Definition \ref{approx-KKT-soln} in Section \ref{alg-constr}. Since $F^*$ is typically unknown, there is a lack of a verifiable termination criterion 
for the methods \cite{AyGa13,NePaGl19,PNT17} to find an $\varepsilon$-gap solution of 
\eqref{conic-p} in general. In contrast, $\varepsilon$-KKT type of solutions can generally be verified and the methods  \cite{LaMo13,Lan16,monteiro2011complexity} are equipped with a usually verifiable termination criterion for finding an $\varepsilon$-KKT type solution of the aforementioned special cases of \eqref{conic-p}.   

In addition to the above methods,  a first-order proximal AL method was recently proposed in \cite[Algorithm 2]{lu2018iteration} for solving a special case of problem \eqref{conic-p} with $P$ having a \emph{compact} domain and $\nabla f$ and $\nabla g$ being \emph{Lipschitz} continuous on $\dom(P)$. At each iteration, this method applies a variant of Nesterov's optimal first-order method \cite[Algorithm 3]{lu2018iteration} to approximately solve a proximal AL subproblem and then updates the Lagrangian multiplier by a classical scheme. This method enjoys two nice features: (i) it is equipped with a verifiable termination criterion; (ii) it achieves a best-known operation complexity of  $\cO(\varepsilon^{-1}\log \varepsilon^{-1})$ for finding an $\varepsilon$-KKT solution\footnote{An $\varepsilon$-KKT solution of \eqref{conic-p} is generally an $\cO(\varepsilon)$-gap solution of \eqref{conic-p} (see Theorems 3 and 6 of \cite{lu2018iteration}). However, the converse may not be true.} of such a special case of \eqref{conic-p}. 


It shall be mentioned that the aforementioned methods in \cite{AyGa13,LaMo13,Lan16,LLM17,lu2018iteration,monteiro2011complexity,NePaGl19,PNT17}   
 and their analysis rely on \emph{boundedness} of $\dom(P)$ and/or \emph{Lipschitz} continuity of $\nabla f$ and $\nabla g$ on $\cl(\dom(P))$ or $\rr^n$. Indeed, these methods use the APG method \cite{Nest13} or its variant as a subproblem solver. Based on the above discussion, such a subproblem solver is not applicable or lacks complexity guarantees in general when $\dom(P)$ is unbounded or $\nabla f$ and $\nabla g$ are merely locally Lipschitz continuous on $\cl(\dom(P))$, because the gradient of the smooth component in the objective function of the subproblems is merely locally Lipschitz continuous. As a result,  these methods are not applicable or lack complexity guarantees  in general when $\dom(P)$ is \emph{unbounded} or $\nabla f$ and $\nabla g$ are merely \emph{locally Lipschitz} continuous on $\cl(\dom(P))$.

In this paper we propose a first-order proximal AL method for solving problem \eqref{conic-p} by following the same framework as \cite[Algorithm 2]{lu2018iteration} except that the proximal AL subproblems are approximately solved by our APG method, namely, Algorithm~\ref{alg-acc-term}. Though the gradient of the smooth component in the objective function of these subproblems is merely locally Lipschitz continuous,  their approximate solutions can be found by our APG method with complexity guarantees. As a result, our  first-order proximal AL method overcomes the aforementioned issue faced by the methods in \cite{AyGa13,LaMo13,Lan16,LLM17,lu2018iteration,monteiro2011complexity,NePaGl19,PNT17}.  
Besides, our method is equipped with a \emph{verifiable} termination criterion and  \emph{almost parameter-free} except that the knowledge on convexity parameter of $f$ is required. Moreover, we show that it achieves an operation complexity of  $\cO(\varepsilon^{-1}\log \varepsilon^{-1})$ and $\cO(\varepsilon^{-1/2}\log \varepsilon^{-1})$ for finding an $\varepsilon$-KKT solution of \eqref{conic-p} when $f$ is convex and strongly convex, respectively. 

The main contributions of our paper are summarized as follows.

\bi
\item 
We propose and analyze APG methods for solving problem \eqref{unc-prob} under \emph{local Lipschitz} continuity of $\nabla f$ on $\cl(\dom(P))$ for the first time. Our proposed methods are almost parameter-free, equipped with a verifiable termination criterion, and enjoy 
an operation complexity of $\cO(\varepsilon^{-1/2}\log \varepsilon^{-1})$ and $\cO(\log \varepsilon^{-1})$ for finding an $\varepsilon$-residual solution of \eqref{unc-prob} when $f$ is convex and strongly convex, respectively.
\item 
We propose and analyze a first-order proximal AL method for solving problem \eqref{conic-p} under \emph{local Lipschitz} continuity of $\nabla f$ and $\nabla g$ on $\cl(\dom(P))$ and possible \emph{unboundedness} of $\dom(P)$ for the first time. Our proposed method is almost parameter-free, equipped with a verifiable termination criterion, and enjoys 
an operation complexity of $\cO(\varepsilon^{-1}\log \varepsilon^{-1})$ and $\cO(\varepsilon^{-1/2}\log \varepsilon^{-1})$ for finding an $\varepsilon$-KKT solution of \eqref{conic-p} when $f$ is convex and strongly convex, respectively. 
\ei

The rest of this paper is organized as follows. In Subsection \ref{notation} we introduce some notation and terminology. In Section~\ref{alg-unconstr} we propose accelerated proximal gradient methods for problem \eqref{unc-prob} and study their worst-case complexity. In Section~\ref{alg-constr} we propose a first-order proximal augmented Lagrangian method for problem \eqref{conic-p} and study its worst-case complexity. In addition, we present some preliminary numerical results and the proofs of the main results in Sections \ref{sec:exp} and \ref{sec:proof}. 
Finally, we make some concluding remarks in Section \ref{sec:conclude}.

\subsection{Notation and terminology} \label{notation}
The following notation will be used throughout this paper. Let $\rr^n$ denote the Euclidean space of dimension $n$, $\langle\cdot,\cdot\rangle$ denote the standard inner product, and $\|\cdot\|$ stand for the Euclidean norm or its induced matrix norm. For any $\omega\in\rr$, let $\omega_+=\max\{\omega,0\}$ and $\lceil \omega \rceil$ denote the least integer number greater than or equal to $\omega$. Let $\bbZ_+$ denote the set of positive integers. For any $t,M\in\bbZ_+$, ${\rm mod}(t,M)$ denotes the remainder of $t$ when divided by $ M$.

For a closed convex function $P:\rr^n\rightarrow(-\infty,\infty]$, let $\partial P$ and $\dom(P)$ denote the subdifferential and domain of $P$, respectively.  The proximal operator associated with $P$ is denoted by  
$\prox_P$,  that is,
\begin{equation*}
\label{eq:def-prox}
\prox_P(z) = \argmin_{x\in\rr^n} \left\{ \frac{1}{2}\|x - z\|^2 + P(x) \right\} \quad \forall z \in \rr^n.
\end{equation*}
Since evaluation of $\prox_{\gamma P}(z)$ is often as cheap as that of $\prox_P(z)$, we count evaluation of $\prox_{\gamma P}(z)$ as one evaluation of proximal operator of $P$ for any $\gamma>0$ and $z\in\rr^n$. For a mapping $h:\rr^n \to \rr^l$, $\nabla h$ denotes the transpose of the Jacobian of $h$. $\nabla h$ is called $L$-Lipschitz continuous on a set $\Omega$ for some constant $L>0$ if $\|\nabla h(x)-\nabla h(y)\| \le L\|x-y\|$ for all $x,y\in\Omega$. In addition, $\nabla h$ is called \emph{locally Lipschitz continuous} on $\Omega$ if for any $x\in\Omega$, there exist $L_x>0$ and an open set $\cU_x$ containing $x$ such that $\nabla h$ is $L_x$-Lipschitz continuous on $\cU_x$. 

Given a nonempty closed convex set $\cC\subseteq\rr^n$, $\dist(x,\cC)$ stands for the Euclidean distance from $x$ to $\cC$, and $\Pi_\cC(x)$ denotes the Euclidean projection of $x$ onto $C$.
The normal cone of $\cC$ at any $x\in \cC$ is denoted by $\mcN_\cC(x)$. For a closed convex cone $\mcK\subseteq \rr^m$, we use $\mcK^*$ to denote the dual cone of $\mcK$, that is,  $\mcK^* = \{y\in\rr^m: \langle y,x\rangle\geq 0, \; \forall x\in\mcK\}$.

\section{Accelerated proximal gradient methods for unconstrained convex optimization}
\label{alg-unconstr}


In this section we consider problem \eqref{unc-prob} and propose accelerated proximal gradient (APG) methods for solving it. In particular, we aim to find an \emph{$\epsilon$-residual solution} of \eqref{unc-prob}, which is defined below.

\begin{defi} \label{approx-residual-soln}
Given any $\epsilon>0$, we say $x\in\rr^n$ is an $\epsilon$-residual solution of problem \eqref{unc-prob} if it satisfies $\dist(0,\partial F(x))\leq \epsilon$.
\end{defi}

To proceed, let $\mu\geq0$ denote the  \emph{convexity parameter} of $f$ on $\dom( P)$, that is, 
\begin{equation}\label{convex-1}
f(y)\geq f(x)+\langle\nabla f(x), y-x\rangle+\frac{\mu}{2}\|x-y\|^2,\quad\forall x\in \dom( P), y\in \rr^n.
\end{equation}
Clearly, $f$ is strongly convex on $\dom( P)$ when $\mu>0$. 
In addition, we assume that the proximal operator associated with $P$ can be exactly evaluated and problem \eqref{unc-prob} has at least one optimal solution. Let $x^*$ be an arbitrary optimal solution of \eqref{unc-prob} and fixed throughout this section.

\subsection{An APG method without a termination criterion for problem \eqref{unc-prob}}
\label{alg-unconstr1}

We propose an APG method for \eqref{unc-prob} as follows, which is a modification of \cite[Algorithm 1 with a single block]{LiLuXi14}  by incorporating a backtracking line search scheme and an adaptive update strategy on the algorithm parameters. 

\begin{algorithm}[H]
\caption{An APG method without a termination criterion for problem \eqref{unc-prob}}
\label{alg-acc}
\begin{algorithmic}[1]
\REQUIRE  $\gamma_0\in(0,1/\mu]$,\footnote{} $0<\alpha_0\in[\sqrt{\mu\gamma_0},1]$, $\delta\in(0,1)$, and $x^1=z^1\in \dom( P)$.
\FOR{$t=1,2,\dots$}
\STATE Compute
\begin{align}
\label{unc-iter}
&y^t=\left((1-\alpha_t)x^t+\alpha_t(1-\beta_t)z^t\right)/(1-\alpha_t \beta_t),\\
\label{unc-prox}
&z^{t+1}=\argmin_{x}\left\{\gamma_t[\langle\nabla f(y^t),x\rangle+ P(x)]+\frac{\alpha_t}{2}\|x-\beta_ty^t-(1-\beta_t)z^t\|^2\right\},\\
\label{unc-updx}
&x^{t+1}=(1-\alpha_t)x^t+\alpha_tz^{t+1},
\end{align}
where $\gamma_t=\gamma_0\delta^{n_t}$ and $\beta_t=\mu\gamma_t\alpha_t^{-1}$ with $\alpha_t\in(0,1]$ being the solution of
\begin{equation}\label{unc-equa}
\gamma_{t-1}\alpha_t^2=(1-\alpha_t)\alpha_{t-1}^2\gamma_t+\mu\alpha_t\gamma_t\gamma_{t-1},
\end{equation}
and $n_t$ being the smallest non-negative integer such that
\begin{equation}\label{unc-line}
2\gamma_t\left(f(x^{t+1})-f(y^t)-\langle\nabla f(y^t),x^{t+1}-y^t\rangle\right)\leq\|x^{t+1}-y^t\|^2.
\end{equation}
\ENDFOR
\end{algorithmic}
\end{algorithm}
\footnotetext{By convention, we define $1/0=\infty$. Consequently, when $\mu=0$, $\gamma_0$ can be any positive number.}

\begin{rem}
(i) Algorithm \ref{alg-acc} is almost parameter-free except that the convexity parameter $\mu$ of $f$ is required.

(ii) 
One can observe that the fundamental operations of Algorithm \ref{alg-acc} consist of evaluations of $\nabla f$ and proximal operator of $P$. Specifically, at iteration $t$, Algorithm \ref{alg-acc} requires $n_t+1$ evaluations of $\nabla f$ and proximal operator of $P$ for finding $x^{t+1}$ satisfying \eqref{unc-line}.  

(iii) Notice from Algorithm \ref{alg-acc} that $0<\alpha_0\in[\sqrt{\mu\gamma_0},1]$, which implies $\alpha_0\in(0,1]$ regardless of $\mu=0$ or $\mu>0$.
Suppose that $\alpha_{t-1}\in(0,1]$ and $\gamma_{t-1}, \gamma_t\in (0, \gamma_0]$ are given  for some $t \geq 1$. Then $\alpha_t\in(0,1]$ is well defined by the equation \eqref{unc-equa}. Indeed, let  
$\phi(\alpha)=\gamma_{t-1}\alpha^2-(1-\alpha)\alpha_{t-1}^2\gamma_t-\mu\alpha\gamma_t\gamma_{t-1}$.
Observe that $\phi(0)=-\alpha_{t-1}^2\gamma_t<0$ and $\phi(1)=\gamma_{t-1}(1-\mu\gamma_t)\geq\gamma_{t-1}(1-\mu\gamma_0)\geq0$ due to $\gamma_0\in(0,1/\mu]$. Hence, \eqref{unc-equa} has a solution in $(0,1]$ and $\alpha_t$ is well-defined.
\end{rem}

We next study \emph{well-definedness} of Algorithm \ref{alg-acc} and also its \emph{convergence rate} in terms of $F(x^t)-F(x^*)$. 
To proceed,  we define
\begin{equation}\label{def-S}
 r_0=\sqrt{F(x^1)-F(x^*)+\frac{\alpha^2_{0}}{2\gamma_{0}}\|x^1-x^*\|^2},\;\;\cS=\left\{x\in \dom( P):\|x-x^*\| \leq\frac{\sqrt{2\gamma_0} r_0}{\alpha_0}\right\}.
\end{equation}

The following lemma establishes that $\nabla f$ is \emph{Lipschitz} continuous on $\cS$ and also on an enlarged set induced by $\alpha_0$, $\gamma_0$, $r_0$, $x^*$, 
 $f$ and $\cS$, albeit $\nabla f$ is \emph{locally Lipschitz} continuous on $\cl(\dom(P))$. This result will play an important role in this section.

\begin{lemma} \label{F-Lipschitz}
Let $r_0$ and $\cS$ be defined in \eqref{def-S}, and let $\gamma_0$ and $\alpha_0$ be the input parameters of Algorithm \ref{alg-acc}. Then the following statements hold.
\begin{enumerate}[label=(\roman*)]
\item $\nabla f$ is $L_\cS$-Lipschitz continuous on $\cS$ for some constant $L_\cS>0$.
\item $\nabla f$ is $L_\chS$-Lipschitz continuous on $\chS$ for some constant $L_\chS>0$, where
\begin{equation}\label{def-hS}
\chS=\left\{x\in \dom( P):\|x-x^*\| \leq\left(1+\gamma_0L_\cS\right)\frac{\sqrt{2\gamma_0} r_0}{\alpha_0}\right\}.
\end{equation}
 \end{enumerate}
\end{lemma}

\begin{proof}
Notice that $\cS$ is a convex and bounded subset in $\dom( P)$. By this and the local Lipschitz continuity of $\nabla f$ on $\cl(\dom( P))$, it is not hard to observe that there exists some constant  $L_\cS>0$ such that $\nabla f$ is $L_\cS$-Lipschitz continuous on $\cS$. Hence, statement (i) holds and moreover the set $\chS$ is well-defined. By a similar argument, one can see that statement (ii) also holds.
\end{proof}

The following theorem shows that Algorithm \ref{alg-acc} is \emph{well-defined} at each iteration.  Its proof is deferred to Subsection \ref{sec:proof1}.

\begin{thm}\label{inner} 
Algorithm~\ref{alg-acc} is well-defined at each iteration. Moreover, $x^t,y^t, z^t\in\cS$ and  $n_t\leq N$ for all $t\geq 1$, where $\cS$ is defined in \eqref{def-S} and
\begin{equation}\label{def-N}
N=\left\lceil\frac{\log(\gamma_0L_\chS)}{\log(1/\delta)}\right\rceil_+.
\end{equation}
\end{thm}

The next theorem presents a result regarding \emph{convergence rate} of Algorithm~\ref{alg-acc}, whose proof is deferred to Section \ref{sec:proof}.

\begin{thm}\label{t0} 
Let $\{x^t\}$ be generated by Algorithm~\ref{alg-acc}.  
Then for all  $t \ge 1$, it holds that 
\beq \label{opt-gap}
F(x^t)-F(x^*) \leq\min\left\{\left(1-\sqrt{\mu\min\left\{\gamma_0,\delta L_\chS^{-1}\right\}}\ \right)^{t-1},\; 4\left(2+(t-1)\alpha_0\sqrt{\min\left\{1,\delta \gamma_0^{-1}L_\chS^{-1}\right\}}\ \right)^{-2}\right\}r^2_0.
\eeq
\end{thm}

\begin{rem} \label{remark-alg1}
(i) Despite only assuming local Lipschitz continuity of $\nabla f$ on $\cl(\dom(P))$, Algorithm~\ref{alg-acc} enjoys a similar convergence rate as the optimal APG method \cite[Algorithm 1 with a single block]{LiLuXi14} which was proposed and analyzed for solving 
a special case of problem \eqref{unc-prob} with $\nabla f$ being Lipschitz continuous on $\rr^n$.

(ii) An adaptive gradient method was recently proposed in \cite[Algorithm 1]{mali20} for solving a special case of problem \eqref{unc-prob} with $P\equiv 0$. It is a variant of classical gradient methods without acceleration and enjoys a much worse convergence rate than the one given in \eqref{opt-gap}. In particular, when $f$ is convex, it has a convergence rate of $\cO(1/t)$ (see \cite[Theorem 1]{mali20}).
\end{rem}

From theoretical perspective, it follows from Theorem \ref{t0} that Algorithm~\ref{alg-acc} enjoys an iteration complexity of $\cO(\varepsilon^{-1/2})$ and $\cO(\log\varepsilon^{-1})$ for finding an $\varepsilon$-gap solution $x^t$ of \eqref{unc-prob} satisfying $F(x^t)-F^* \le \varepsilon$ when $f$ is convex and strongly convex, respectively.  However, since $F^*$,  $L_\chS^{-1}$ and $r_0$ are typically unknown,  it is difficult to come up with a verifiable termination criterion for Algorithm~\ref{alg-acc} to find an $\varepsilon$-gap solution of \eqref{unc-prob}. To circumvent this issue, we propose some variants of Algorithm~\ref{alg-acc} with a verifiable termination criterion in the next two subsections.

\subsection{An APG method with a termination criterion for problem \eqref{unc-prob} with $\mu>0$} \label{alg-unconstr2}

In this subsection we propose an APG method with a \emph{verifiable termination criterion} for finding an \emph{$\epsilon$-residual solution} of problem \eqref{unc-prob} with $\mu>0$, namely, $f$ being strongly convex on $\dom(P)$. It is a slight variant of Algorithm~\ref{alg-acc} by incorporating a termination criterion that is checked only \emph{periodically}. 

\begin{algorithm}[H]
\caption{An APG method with a termination criterion for problem \eqref{unc-prob} with $\mu>0$}
\label{alg-acc-term}
\begin{algorithmic}[1]
\REQUIRE $\epsilon>0$, $\gamma_0\in(0,1/\mu]$, $0<\alpha_0\in[\sqrt{\mu\gamma_0},1]$, $\delta\in(0,1)$, $M\in\bbZ_+$, and $x^1=z^1\in \dom( P)$.
\FOR{$t=1,2,\dots$}
\STATE Compute
\begin{align*}
&y^t=\left((1-\alpha_t)x^t+\alpha_t(1-\beta_t)z^t\right)/(1-\alpha_t \beta_t),\\
&z^{t+1}=\argmin_{x}\left\{\gamma_t[\langle\nabla f(y^t),x\rangle+ P(x)]+\frac{\alpha_t}{2}\|x-\beta_ty^t-(1-\beta_t)z^t\|^2\right\},\\
&x^{t+1}=(1-\alpha_t)x^t+\alpha_tz^{t+1},
\end{align*}
where $\gamma_t=\gamma_0\delta^{n_t}$ and $\beta_t=\mu\gamma_t\alpha_t^{-1}$ with $\alpha_t\in(0,1]$ being the solution of
\[
\gamma_{t-1}\alpha_t^2=(1-\alpha_t)\alpha_{t-1}^2\gamma_t+\mu\alpha_t\gamma_t\gamma_{t-1},
\]
and $n_t$ being the smallest non-negative integer such that
\begin{equation*}
2\gamma_t\left(f(x^{t+1})-f(y^t)-\langle\nabla f(y^t),x^{t+1}-y^t\rangle\right)\leq\|x^{t+1}-y^t\|^2.
\end{equation*}
\IF {$\mathrm{mod}(t,M)=0$}
\STATE Call Algorithm~\ref{alg-term} with $(x^{t+1},\gamma_0,\delta)$ as the input and output $(\tx^{t+1},\tg_{t+1})$.
\STATE Terminate the algorithm and output $\tx^{t+1}$ if 
\begin{equation}\label{unc-check3}
\|\tg_{t+1}^{-1}(x^{t+1}-\tx^{t+1})+\nabla f(\tx^{t+1})-\nabla f(x^{t+1})\|\leq\epsilon.
\end{equation}
\ENDIF
\ENDFOR
\end{algorithmic}
\end{algorithm}


\begin{algorithm}[H]
\caption{Adaptive proximal gradient iteration}
\label{alg-term}
\begin{algorithmic}[1]
\REQUIRE $v\in\cS$ and $\tg_0, \delta>0$.
\STATE Compute
\begin{equation}\label{unc-xplus}
\tv=\argmin_{x}\left\{\tg\langle\nabla f(v),x\rangle+\tg P(x)+\frac{1}{2}\|x-v\|^2\right\},
\end{equation}
where $\tg=\tg_0\delta^{\tn}$ with $\tn$ being the smallest non-negative integer such that
\beq
2\tg(f(\tv)-f(v)-\langle\nabla f(v),\tv-v\rangle)\leq\|\tv-v\|^2. \label{unc-check1}
\eeq
\STATE  Terminate the algorithm and output $(\tv,\tg)$.
\end{algorithmic}
\end{algorithm}

\begin{rem}
It is clear to see that Algorithm \ref{alg-acc-term} is well-defined at each iteration and equipped with a verifiable termination criterion. In addition, it is almost parameter-free except that the convexity parameter $\mu$ of $f$ is required.
\end{rem}

The following theorem presents an \emph{iteration and operation complexity} of Algorithm~\ref{alg-acc-term} for finding an $\epsilon$-residual solution of problem \eqref{unc-prob} with a strongly convex $f$ on $\dom(P)$, whose proof is deferred to Subsection \ref{sec:proof2}.

\begin{thm}\label{res-complexity} 
Suppose that $\mu>0$, i.e., $f$ is strongly convex on $\dom(P)$. Let $\epsilon$, $M$, $\delta$, $\alpha_0$ and $\gamma_0$ be the input parameters of Algorithm~\ref{alg-acc-term},  $r_0$ and $L_\chS$ be given in \eqref{def-S} and Lemma \ref{F-Lipschitz} respectively, and let
\begin{align}
T &=  M+\left\lceil\frac{2\log\frac{\epsilon}{r_0\left(\sqrt{2\max\{\gamma_0^{-1},L_\chS\delta^{-1}\}}+\sqrt{2\gamma_0}L_\chS\right)}}{\log\left(1-\sqrt{\mu\min\left\{\gamma_0,\delta L_\chS^{-1}\right\}}\,\right)}\right\rceil_+,  \label{def-T} \\ 
\bar N & =  (1+M^{-1})\left(M+\left\lceil\frac{2\log\frac{\epsilon}{r_0\left(\sqrt{2\max\{\gamma_0^{-1},L_\chS\delta^{-1}\}}+\sqrt{2\gamma_0}L_\chS\right)}}{\log\left(1-\sqrt{\mu\min\left\{\gamma_0,\delta L_\chS^{-1}\right\}}\,\right)}\right\rceil_+\right)\left(1+ \left\lceil\frac{\log(\gamma_0L_\chS)}{\log(1/\delta)}\right\rceil_+\right). \label{def-NN}
\end{align}
Then Algorithm~\ref{alg-acc-term} terminates and outputs an $\epsilon$-residual solution of problem \eqref{unc-prob} in at most $T$ iterations. Moreover, the total number of evaluations of $\nabla f$ and proximal operator of $P$ performed in Algorithm~\ref{alg-acc-term} is no more than $\bar N$, respectively.
\end{thm}

\begin{rem}
It can be seen from Theorem \ref{res-complexity} that Algorithm~\ref{alg-acc-term} enjoys an operation complexity  of $\cO(\log\varepsilon^{-1})$ for finding an $\epsilon$-residual solution of problem \eqref{unc-prob} with a strongly convex $f$ on $\dom(P)$.
\end{rem}

\subsection{An APG method with a termination criterion for problem \eqref{unc-prob} with $\mu=0$} \label{alg-unconstr3}

In this subsection we propose an APG method with a \emph{verifiable termination criterion} for finding an \emph{$\varepsilon$-residual solution} of problem \eqref{unc-prob} with $\mu=0$, namely, $f$ being  convex but not strongly convex on $\dom(P)$. In particular, the proposed APG method applies Algorithm~\ref{alg-acc-term} to a sequence of strongly convex optimization problems arising from a perturbation of problem \eqref{unc-prob}.

\begin{algorithm}[H]
\caption{An APG method with a termination criterion for problem \eqref{unc-prob} with $\mu=0$}
\label{PPA-sp}
\begin{algorithmic}[1]
\REQUIRE $\varepsilon>0$, $x_0\in\dom(P)$, $M\in \bbZ_+$,  $0<\delta<1$, $\rho_0> 1$, $0<\gamma_0\leq\rho_0$, $\alpha_0\in [\sqrt{\gamma_0/\rho_0},1]$, $0<\eta_0\leq1$, $\zeta>1$, $0<\sigma<1/\zeta$, $\rho_k=\rho_0\zeta^k$, $\eta_k=\eta_0\sigma^k$ for all $k \ge 0$.
\FOR{$k=0,1,\dots$}
\STATE Call Algorithm~\ref{alg-acc-term} with $F \leftarrow F_k$, $f \leftarrow f_k$, $\epsilon \leftarrow\eta_k$, $\mu \leftarrow \rho_k^{-1}$  , $x^1=z^1 \leftarrow x^k$ and the parameters $\alpha_0$, $\gamma_0$, $\delta$ and $M$, and denote its output by $x^{k+1}$, 
where
\begin{align}
f_k(x) = f(x)+\frac{1}{2\rho_k}\|x-x^k\|^2, \quad F_k(x) = f_k(x)+P(x). \label{fk-sp} 
\end{align}
\STATE Terminate the algorithm and output $x^{k+1}$  if  
\begin{align}
\frac{1}{\rho_k}\|x^{k+1}-x^k\|\leq\frac{\epsilon}{2},\ \eta_k\leq\frac{\epsilon}{2}.\label{ppa-term-sp}
\end{align}
\ENDFOR
\end{algorithmic}
\end{algorithm}

\begin{rem}
Algorithm \ref{PPA-sp} is parameter-free and equipped with a verifiable termination criterion. In addition, by the monotonicity of $\{\rho_k\}$, one has
\begin{align*}
0<\gamma_0\leq\rho_0\leq \rho_k,\qquad
\sqrt{\rho_k^{-1}\gamma_0}\leq\sqrt{\rho_0^{-1}\gamma_0}\leq\alpha_0\leq 1.
\end{align*}
Consequently, the choice of $\alpha_0$ and $\gamma_0$ in Algorithm \ref{PPA-sp} satisfies the requirements specified in Algorithm~\ref{alg-acc-term}. It then follows from Theorem \ref{res-complexity} that at the $k$th outer iteration of Algorithm \ref{PPA-sp}, $x^{k+1}$ must be successfully generated by Algorithm~\ref{alg-acc-term}, which is an $\eta_k$-residual solution of the problem $\min_x \{F_k(x)=f_k(x)+P(x)\}$. Thus,  it holds that 
\begin{equation}\label{subprob-term-sp}
\dist(0,\partial F_k(x^{k+1}))\leq\eta_k.
\end{equation}  
\end{rem}

We next study iteration and operation complexity of Algorithm \ref{PPA-sp} for finding an $\varepsilon$-residual solution of problem \eqref{unc-prob} with $f$ being convex but not strongly convex on $\dom(P)$. Before proceeding, we introduce some notation that will be used subsequently.
We define
\begin{align}
r_0 &=\|x^0-x^*\|,\quad \theta=\sum_{i=0}^\infty\rho_i\eta_i=\frac{\rho_0\eta_0}{1-\sigma\zeta},  \label{def1-sp} \\
\tr0 &=\max\left\{\sqrt{2\gamma_0\alpha_0^{-2}(F(x^0)-F(x^*))+r_0^2}, \  \sqrt{2\gamma_0\alpha_0^{-2}(r_0+\theta)\left(\eta_0+\rho_0^{-1}(r_0+\theta)\right)+(r_0+\theta)^2}\right\}. \label{def-tP-sp}
\end{align}
Also, we define 
\begin{equation}\label{def-P-sp}
\mcQ=\left\{x\in \dom( P):\|x-x^*\|\leq\tr0+r_0+\theta\right\}.
\end{equation}
Let $L_{\nabla f}$ be the Lipschitz constant of $\nabla f$ on $\mcQ$
and 
\begin{equation}\label{def-Q-sp}
L = L_{\nabla f}+\rho_0^{-1},\quad \chQ=\left\{x\in \dom( P):\|x-x^*\| \leq (1+\gamma_0L)\tr0+r_0+\theta\right\}, \quad \tL=\tL_{\nabla f}+\rho_0^{-1}, 
\end{equation}
where $\tL_{\nabla f}$ is the Lipschitz constant of $\nabla f$ on $\chQ$. By the local Lipschitz continuity of $\nabla f$ on $\cl(\dom(P))$ and a similar argument as in the proof of Lemma \ref{F-Lipschitz}, one can easily observe that $L$, $\tL$, $L_{\nabla f}$, $\tL_{\nabla f}$, $\mcQ$, and $\chQ$ are well-defined.

The following theorem presents an\emph{ iteration and operation complexity} of Algorithm~\ref{PPA-sp} for finding an \emph{$\epsilon$-residual solution} of problem \eqref{unc-prob} with $f$ being convex but not strongly convex on $\dom(P)$, namely, a point $x$ satisfying $\dist(0,\partial F(x))\leq \varepsilon$, whose proof is deferred to Subsection \ref{sec:proof3}.

\begin{thm} \label{thm:cvx-unconstr}  
Suppose that $\mu=0$, i.e., $f$ is convex but not strongly convex on $\dom(P)$. Let $\varepsilon$, $M$, $\delta$, $\rho_0$, $\alpha_0$, $\gamma_0$,  $\eta_0$, $\zeta$ and $\sigma$ be the input parameters of Algorithm~\ref{PPA}, and let $r_0$, $\theta$, $\tr0$ and $\tL$ be  given in \eqref{def1-sp}, \eqref{def-tP-sp} and \eqref{def-Q-sp}, respectively. Define
\begin{align}
\tC_1  &= (1+M^{-1})\left(1+\left\lceil\frac{\log(\gamma_0\tL)}{\log(1/\delta)}\right\rceil_+\right), \label{C1-sp} \\
\tC_2 &=  \frac{\sqrt{\rho_0\zeta} \tC_1\left(\log\frac{\alpha_0^2\tr0^2\left(\sqrt{\max\{\gamma_0^{-2},\gamma_0^{-1}\tL\delta^{-1}\}}+\tL\right)^2}{\eta_0^2}\right)_+}{(\sqrt{\zeta}-1)\min\left\{\sqrt{\gamma_0},\sqrt{\delta \tL^{-1}}\right\}}, \quad 
\tC_3 =\frac{2\sqrt{\rho_0\zeta}\tC_1\log(1/\sigma)}{(\sqrt{\zeta}-1)\min\left\{\sqrt{\gamma_0},\sqrt{\delta \tL^{-1}}\right\}}. \label{C2-sp}
\end{align}
Then the following statements hold.
\bi
\item[(i)]
Algorithm~\ref{PPA-sp} outputs an $\varepsilon$-residual solution of problem \eqref{unc-prob} after at most $K+1$ outer iterations, where 
\beq \label{K-sp}
K =\left\lceil\max\left\{\log\left(\frac{2r_0+2\theta}{\rho_0\varepsilon}\right)/\log\zeta,\frac{\log(2\eta_0/\varepsilon)}{\log(1/\sigma)}\right\}\right\rceil_+. 
\eeq 

\item[(ii)]
The total number of evaluations of $\nabla f$ and proximal operator of $P$ performed in Algorithm~\ref{PPA-sp} is no more than $\widetilde N$, respectively, where
\begin{align}
\widetilde N = & \ (M+1)\tC_1+ (M+1) \tC_1\left\lceil\max\left\{\log\left(\frac{2r_0+2\theta}{\varepsilon\rho_0}\right)/\log\zeta,\frac{\log(2\eta_0/\varepsilon)}{\log(1/\sigma)}\right\}\right\rceil_+\notag \\
&+\tC_2 \max\left\{\sqrt{\frac{2\zeta(r_0+\theta)}{\varepsilon\rho_0}},\sqrt{\zeta}\left(\frac{2\eta_0}{\varepsilon}\right)^{\frac{\log\zeta}{2\log(1/\sigma)}},1\right\} \notag \\
&+\tC_3\left\lceil\max\left\{\log\left(\frac{2r_0+2\theta}{\varepsilon\rho_0}\right)/\log\zeta,\frac{\log(2\eta_0/\varepsilon)}{\log(1/\sigma)}\right\}\right\rceil_+  \max\left\{\sqrt{\frac{2\zeta(r_0+\theta)}{\varepsilon\rho_0}},\sqrt{\zeta}\left(\frac{2\eta_0}{\varepsilon}\right)^{\frac{\log\zeta}{2\log(1/\sigma)}},1\right\}.\label{complexity-sp}
\end{align}
\ei
\end{thm} 

\begin{rem}
Since $1<\zeta<1/\sigma$, it can be seen from Theorem \ref{thm:cvx-unconstr}  that Algorithm~\ref{PPA-sp} enjoys an operation complexity of $\cO(\varepsilon^{-1/2}\log \varepsilon^{-1})$ for finding an $\varepsilon$-residual solution of problem \eqref{unc-prob} with $f$ being convex but not strongly convex on $\dom(P)$.
\end{rem}

\section{A first-order proximal augmented Lagrangian method for constrained convex optimization}
\label{alg-constr}

In this section we consider problem \eqref{conic-p} and propose a first-order proximal augmented Lagrangian (AL) method for solving it.  Let $\mu\geq0$ denote the \emph{convexity parameter} of $f$ on $\dom( P)$, that is, \eqref{convex-1} holds for $f$ and $\mu$. Before proceeding, we make the following additional assumptions for problem \eqref{conic-p}.

\begin{ass} \label{assump-conic}
\begin{enumerate}[label=(\alph*)]
\item The proximal operator associated with $P$ and the projection onto $\mcK^*$ can be exactly evaluated.
\item Both problem \eqref{conic-p} and its Lagrangian dual problem
\begin{align}\label{conic-d}
\sup_{\lambda\in\mcK^*}\inf_{x}\left\{f(x)+P(x)+\langle\lambda,g(x)\rangle\right\}
\end{align}
have optimal solutions, and moreover, they share the same optimal value.
\end{enumerate}
\end{ass}

Under the assumptions on problem \eqref{conic-p}, it can be observed that $(x,\lambda)$ is a pair of optimal solutions of \eqref{conic-p} and \eqref{conic-d} if and only if it satisfies the Karush-Kuhn-Tucker (KKT) condition
\begin{align*}
0 \in \begin{pmatrix}
\nabla f(x)+\nabla g(x)\lambda+\partial P(x)\\
-g(x)+\mcN_{\mcK^*}(\lambda)
\end{pmatrix}. 
\end{align*}
In general, it is difficult to find an exact optimal solution of \eqref{conic-p} and \eqref{conic-d}. Instead, for any given $\varepsilon>0$, we are interested in finding an $\varepsilon$-KKT solution $(x,\lambda)$ of problems \eqref{conic-p} and \eqref{conic-d} that is defined below.

\begin{defi} \label{approx-KKT-soln}
Given any $\varepsilon>0$, we say $(x,\lambda)\in\rr^n \times \rr^m$ is an $\varepsilon$-KKT solution of problems \eqref{conic-p} and \eqref{conic-d}  if
\begin{equation*}
\dist(0,\nabla f(x)+\partial  P(x)+\nabla g(x)\lambda)\leq \varepsilon,\quad\dist(g(x),\mcN_{\mcK^*}(\lambda))\leq\varepsilon.
\end{equation*}
\end{defi}

We next propose a first-order proximal AL method with a verifiable termination criterion for solving problem \eqref{conic-p}, which follows the same framework as \cite[Algorithm 2]{lu2018iteration} except that the proximal AL subproblems are approximately solved by our newly proposed APG method, namely, Algorithm~\ref{alg-acc-term}. Specifically, at the $k$th iteration, our method applies Algorithm~\ref{alg-acc-term} to approximately solve the proximal AL subproblem 
\[
\min_x \mcL(x,\lambda^k;\rho_k) + \frac{1}{2\rho_k} \|x-x^k\|^2
\]
for some $\lambda^k\in\mcK^*$ and $\rho_k>0$, where $\mcL$ is the AL function associated with problem \eqref{conic-p} defined as
\begin{equation}\label{conic-AL}
\mcL(x,\lambda;\rho)=f(x)+ P(x)+\frac{1}{2\rho}\left(\dist^2\left(\lambda+\rho g(x),-\mcK\right)-\|\lambda\|^2\right). 
\end{equation}

\begin{algorithm}[H]
\caption{A first-order proximal augmented Lagrangian method for problem \eqref{conic-p}}
\label{PPA}
\begin{algorithmic}[1]
\REQUIRE $\varepsilon>0$, $(x_0,\lambda_0)\in\dom(P)\times\mcK^*$, $M\in \bbZ_+$,  $0<\delta<1$, $\rho_0> (\mu+\sqrt{\mu^2+4})/2$, $\alpha_0\in [\sqrt{(\mu+1/\rho_0)/\rho_0},1]$, $0<\eta_0\leq1$, $\zeta>1$, $0<\sigma<1/\zeta$, $\rho_k=\rho_0\zeta^k$, $\eta_k=\eta_0\sigma^k$ for all $k \ge 0$.
\FOR{$k=0,1,\dots$}
\STATE Call Algorithm~\ref{alg-acc-term} with $F \leftarrow F_k$, $f \leftarrow f_k$, $\epsilon \leftarrow\eta_k$, $\gamma_0 \leftarrow \rho_k^{-1}$,   $\mu \leftarrow \mu+\rho_k^{-1}$  , $x^1=z^1 \leftarrow x^k$ and the parameters $\alpha_0$, $\delta$ and $M$, and denote its output by $x^{k+1}$, 
where
\begin{align}
f_k(x) = f(x)+\frac{1}{2\rho_k}\left(\dist^2\left(\lambda^k+\rho_k g(x),-\mcK\right)-\|\lambda^k\|^2+\|x-x^k\|^2\right), \quad F_k(x) = f_k(x)+P(x). \label{fk} 
\end{align}
\label{step2}
\STATE Set $\lambda^{k+1}=\Pi_{\mcK^*}\left(\lambda^k+\rho_kg(x^{k+1})\right)$.
\STATE Terminate the algorithm and output $(x^{k+1},\lambda^{k+1})$  if  
\begin{align}
\frac{1}{\rho_k}\|(x^{k+1},\lambda^{k+1})-(x^k,\lambda^k)\|\leq\frac{\epsilon}{2},\quad  \eta_k\leq\frac{\epsilon}{2}.\label{ppa-term}
\end{align}
\ENDFOR
\end{algorithmic}
\end{algorithm}

\begin{rem}
(i) Algorithm \ref{PPA} is equipped with a verifiable termination criterion and almost parameter-free except that the convexity parameter $\mu$ of $f$ is required.

(ii) Since $\rho_0> (\mu+\sqrt{\mu^2+4})/2$, it follows that $\rho_0^{-1} < 1/(\mu+\rho_0^{-1})$. By this, $\alpha_0\in [\sqrt{(\mu+1/\rho_0)/\rho_0},1]$, and the monotonicity of $\{\rho_k\}$, one has
\begin{align*}
0<\rho_k^{-1} \le \rho_0^{-1} < \frac{1}{\mu+\rho_0^{-1}} \le \frac{1}{\mu+\rho_k^{-1}},  \qquad
\sqrt{(\mu+\rho_k^{-1})\rho_k^{-1}} \le \sqrt{(\mu+\rho_0^{-1})\rho_0^{-1}} \le \alpha_0 \le 1.
\end{align*}
Consequently, the choice of $\alpha_0$ and $\gamma_0$ in Algorithm \ref{PPA} satisfies the requirements specified in Algorithm~\ref{alg-acc-term}. It then follows from Theorem \ref{res-complexity} that at the $k$th outer iteration of Algorithm \ref{PPA}, $x^{k+1}$ must be successfully generated by Algorithm~\ref{alg-acc-term}, which is an $\eta_k$-residual solution of the problem $\min_x \{F_k(x)=f_k(x)+P(x)\}$. Thus, it holds that 
\begin{equation}\label{subprob-term}
\dist(0,\partial F_k(x^{k+1}))\leq\eta_k.
\end{equation}
\end{rem}

\medskip

We next study iteration and operation complexity of Algorithm \ref{PPA} for finding an $\varepsilon$-KKT solution of problems \eqref{conic-p} and \eqref{conic-d}. Before proceeding, we introduce some notation that will be used subsequently.

Let $(x^*,\lambda^*)$ be an arbitrary pair of optimal solutions of problems \eqref{conic-p} and \eqref{conic-d} and fixed throughout this section. We define
\begin{equation}\label{def1}
r_0=\|(x^0,\lambda^0)-(x^*,\lambda^*)\|,\quad \theta=\sum_{i=0}^\infty\rho_i\eta_i=\frac{\rho_0\eta_0}{1-\sigma\zeta},\quad \ctQ=\left\{x\in \dom( P):\|x-x^*\|\leq r_0+\theta\right\}.
\end{equation}

Let $\widetilde L_g$ be the Lipschitz constant of $g$ on $\ctQ$ and
\begin{align}
\tr0=\max\Bigg\{&\sqrt{2\rho_0^{-1}\alpha_0^{-2} (F(x^0)-F(x^*))+\rho_0^{-2}\alpha_0^{-2} (\|\Pi_{\mcK^*}(\lambda^0+\rho_0g(x^0))\|^2+\|\lambda^0-\lambda^*\|^2-\|\lambda^0\|^2)+r_0^2}, \notag \\ 
&\left. \sqrt{2\rho_0^{-1}\alpha_0^{-2}(r_0+\theta)\left(\eta_0+\rho_0^{-1}(r_0+\theta)+2\widetilde L_g(\zeta+1)(\|\lambda^*\|+r_0+\theta)+\rho_0\alpha_0^2(r_0+\theta)\right)}\right\}. \label{def-tP}
\end{align}
We define 
\begin{equation}\label{def-P}
\mcQ=\left\{x\in \dom( P):\|x-x^*\|\leq\tr0+r_0+\theta\right\}.
\end{equation}
Let $L_{\nabla f}$, $L_{\nabla g}$ and $L_g$ be the Lipschitz constants of $\nabla f$, $\nabla g$ and $g$ on $\mcQ$, respectively, and let
\begin{equation}\label{def2}
C=L_{\nabla g}\sup_{x\in\mcQ}\|g(x)\|+L_g^2,\; B=L_{\nabla f}+L_{\nabla g}(\|\lambda^*\|+\sqrt{2}r_0),\;L= C+\rho_0^{-1}B+\rho_0^{-1}L_{\nabla g}\theta+\rho_0^{-2}.
\end{equation}
We define
\begin{equation}\label{def-Q}
\chQ=\left\{x\in \dom( P):\|x-x^*\| \leq (1+L)\tr0+r_0+\theta\right\}.
\end{equation}
Let $\tL_{\nabla f}$, $\tL_{\nabla g}$ and $\tL_g$ be the Lipschitz constants of $\nabla f$, $\nabla g$ and $g$ on $\chQ$, respectively, and let
\begin{equation}\label{def3}
\wC=\tL_{\nabla g}\sup_{x\in\chQ}\|g(x)\|+\tL_g^2,\;\tB=\tL_{\nabla f}+\tL_{\nabla g}(\|\lambda^*\|+\sqrt{2}r_0),\;\tL=\wC+\rho_0^{-1}\tB+\rho_0^{-1}\tL_{\nabla g}\theta+\rho_0^{-2}.
\end{equation}
By the local Lipschitz continuity of $\nabla f$ and $\nabla g$ on $\cl(\dom(P))$ and a similar argument as in the proof of Lemma \ref{F-Lipschitz}, one can easily observe that $\widetilde L_g$, $L_{\nabla f}$, $L_{\nabla g}$, $L_g$, $\tL_{\nabla f}$, $\tL_{\nabla g}$, $\tL_g$, $B$, 
$C$, $L$, $\tB$, $\wC$, $\tL$, $\mcQ$, and $\chQ$ are well-defined.  

The following theorem presents an \emph{iteration and operation complexity} of Algorithm~\ref{PPA} for finding an \emph{$\epsilon$-KKT solution} of problems \eqref{conic-p} and \eqref{conic-d}, whose proof is deferred to Subsection \ref{sec:proof4}.

\begin{thm}\label{thm:constr}  
Let $\varepsilon$, $M$, $\delta$, $\rho_0$, $\alpha_0$,  $\eta_0$, $\zeta$ and $\sigma$ be the input parameters of Algorithm~\ref{PPA}, and let $r_0$, $\theta$, $\tr0$ and $\tL$ be  given in \eqref{def1}, \eqref{def-tP} and \eqref{def3}, respectively. Define
\begin{align}
\wC_1  &= (1+M^{-1})\left(1+\left\lceil\frac{\log\tL}{\log(1/\delta)}\right\rceil_+\right), \label{C1} \\
\wC_2 &=  \frac{\wC_1\left(\log\frac{\rho_0^2\alpha_0^2\tr0^2\left(\sqrt{\max\{1,\tL\delta^{-1}\}}+\tL\right)^2}{\eta_0^2}\right)_+}{(\sqrt{\zeta}-1)\min\left\{1,\sqrt{\delta \tL^{-1}}\right\}}, \quad 
\wC_3=\frac{2\wC_1 \log(\zeta/\sigma)}{(\sqrt{\zeta}-1)\min\left\{1,\sqrt{\delta \tL^{-1}}\right\}}. \label{C2}
\end{align}
Then the following statements hold.
\bi
\item[(i)]
Algorithm~\ref{PPA} outputs an $\varepsilon$-KKT solution of problems \eqref{conic-p} and \eqref{conic-d} after at most $K+1$ outer iterations, where 
\beq \label{K}
K =\left\lceil\max\left\{\log\left(\frac{2r_0+2\theta}{\rho_0\varepsilon}\right)/\log\zeta,\frac{\log(2\eta_0/\varepsilon)}{\log(1/\sigma)}\right\}\right\rceil_+. 
\eeq 
\item[(ii)]
If $\mu=0$, i.e., $f$ is convex but not strongly convex, the total number of evaluations of $\nabla f$, $\nabla g$, proximal operator of $P$ and projection onto $\mcK^*$ performed in Algorithm~\ref{PPA} is no more than $\widehat N$, respectively, where
\begin{align}
\widehat N = & \ 1+(M+1)\wC_1+\left(1+(M+1)\wC_1\right)\left\lceil\max\left\{\log\left(\frac{2r_0+2\theta}{\varepsilon\rho_0}\right)/\log\zeta,\frac{\log(2\eta_0/\varepsilon)}{\log(1/\sigma)}\right\}\right\rceil_+\notag \\
&+ \wC_2\rho_0\zeta\max\left\{\frac{2\zeta(r_0+\theta)}{\varepsilon\rho_0},\zeta\left(\frac{2\eta_0}{\varepsilon}\right)^{\frac{\log\zeta}{\log(1/\sigma)}},1\right\} \notag \\
&+\wC_3\rho_0\zeta\left\lceil\max\left\{\log\left(\frac{2r_0+2\theta}{\varepsilon\rho_0}\right)/\log\zeta,\frac{\log(2\eta_0/\varepsilon)}{\log(1/\sigma)}\right\}\right\rceil_+ \max\left\{\frac{2\zeta(r_0+\theta)}{\varepsilon\rho_0},\zeta\left(\frac{2\eta_0}{\varepsilon}\right)^{\frac{\log\zeta}{\log(1/\sigma)}},1\right\}.\label{def-complexity}
\end{align}
\item[(iii)]
If $\mu>0$, i.e., $f$ is strongly convex, the total number of evaluations of $\nabla f$, $\nabla g$, proximal operator of $P$ and projection onto $\mcK^*$ performed in Algorithm~\ref{PPA} is no more than $\check N$, respectively, where
\begin{align}
\check N = & \ 1+(M+1)\wC_1+\left(1+(M+1)\wC_1\right)\left\lceil\max\left\{\log\left(\frac{2r_0+2\theta}{\varepsilon\rho_0}\right)/\log\zeta,\frac{\log(2\eta_0/\varepsilon)}{\log(1/\sigma)}\right\}\right\rceil_+\notag \\
&+ \wC_2\sqrt{\frac{\rho_0\zeta}{\mu}}\max\left\{\sqrt{\frac{2\zeta(r_0+\theta)}{\varepsilon\rho_0}},\sqrt{\zeta}\left(\frac{2\eta_0}{\varepsilon}\right)^{\frac{\log\zeta}{2\log(1/\sigma)}},1\right\} \notag \\
&+\wC_3\sqrt{\frac{\rho_0\zeta}{\mu}}\left\lceil\max\left\{\log\left(\frac{2r_0+2\theta}{\varepsilon\rho_0}\right)/\log\zeta,\frac{\log(2\eta_0/\varepsilon)}{\log(1/\sigma)}\right\}\right\rceil_+  \max\left\{\sqrt{\frac{2\zeta(r_0+\theta)}{\varepsilon\rho_0}},\sqrt{\zeta}\left(\frac{2\eta_0}{\varepsilon}\right)^{\frac{\log\zeta}{2\log(1/\sigma)}},1\right\}.\label{strong-complexity}
\end{align}
\ei
\end{thm}

\begin{rem}
Since $1<\zeta<1/\sigma$, it can be seen from Theorem \ref{thm:constr} that Algorithm~\ref{PPA} enjoys an operation complexity  of $\cO(\varepsilon^{-1}\log \varepsilon^{-1})$ and $\cO(\varepsilon^{-1/2}\log \varepsilon^{-1})$ for finding an $\varepsilon$-KKT solution of problems \eqref{conic-p} and \eqref{conic-d} when $f$ is convex and strongly convex on $\dom(P)$, respectively. 
\end{rem}

\section{Numerical results}\label{sec:exp}
In this section we conduct some preliminary experiments to test the performance of our proposed method (Algorithm \ref{PPA}), and compare it with a first-order proximal AL method (FPAL) \cite{lu2018iteration}, the forward-reflected-backward splitting method (FRBS) \cite{mali20} and the modified forward-backward splitting method (MFBS) with an Armijo-Goldstein-type stepsize \cite{tseng2000modified}, respectively. All the algorithms are coded in Matlab and all the computations are performed on a desktop with a 3.60 GHz Intel i7-12700K 12-core processor and 32 GB of RAM.

\subsection{Quadratically constrained quadratic programming with box constraints}\label{sec:qcqp-b}
In this subsection we consider quadratically constrained quadratic programming (QCQP) with box constraints
\begin{align}\label{qcqp-b}
\begin{aligned}
\min&\  \frac{1}{2}x^TAx+b^Tx\\
\mbox{s.t.}&\ \frac{1}{2}x^TB_ix+c_i^Tx+d_i\leq0,\quad i=1,\dots,m,\\
&\ -1\leq x_i \leq1, \quad i=1,\dots,n,
\end{aligned}
\end{align}
where $A,B_1,\dots,B_m\in\rr^{n\times n}$ are positive semidefinite matrices, $b, c_1,\dots,c_m\in\rr^{n}$, and $d_1,\dots,d_m\in\rr$.

For each dimension $n$, we set $m=\lceil 0.05n\rceil$ and randomly generate $10$ instances of problem \eqref{qcqp-b}. In particular, we first generate $x^*\in[-1,1]^n$ whose  entries are first independently chosen from the standard normal distribution and then projected to $[-1,1]$, and $\lambda^*\in\rr_+^m$ whose entries are first independently chosen from the normal distribution with mean $1$ and standard deviation $1$ and then projected to $\rr_+$. We then randomly generate an orthogonal matrix $U$ by performing $U=\mathrm{orth}(\mathrm{randn}(n))$, an $n\times n$ diagonal matrix $D$ whose diagonal entries are first independently chosen from the normal distribution with mean $0$ and standard deviation $100$ and then projected to $\rr_+$, and set $A=UDU^T$. Also, we randomly generate an orthogonal matrix $\tU$ by performing $\tU=\mathrm{orth}(\mathrm{randn}(n))$, an $n\times n$ diagonal matrices $\widetilde D$ whose diagonal entries are first independently chosen from the normal distribution with mean $0$ and standard deviation $0.01$ and then projected to $\rr_+$. We set $B_1=\tU\widetilde D\tU^T$, and generate $B_i,\ i=2,\dots, m$ in a similar vein. In addition, we generate $c_i,\ i=1,\dots,m$ independently according the normal distribution with mean $0$ and standard deviation $0.01$. We finally  choose $b$ and $d_i,\ i=1,\dots,m$ so that the KKT conditions of \eqref{qcqp-b} are satisfied at $(x^*, \lambda^*)$, namely $(x^*, \lambda^*)$ is a KKT point of \eqref{qcqp-b}.

Notice that \eqref{qcqp-b} is a special case of \eqref{conic-p} with $f(x)=x^TAx/2+b^Tx$, $P(x)=\mcI_{[-1,1]^n}(x)$, $g_i(x)=x^TB_ix/2+c_i^Tx+d_i,\ i=1,\dots,m$, and $\mcK=\rr_+^m$, where $\mcI_{[-1,1]^n}(\cdot)$ is the indicator function of $[-1,1]^n$. Moreover, $f$ and $P$ are convex, $g$ is $\mcK$-convex, $\dom(P)$ is compact, and $\nabla f$ and $\nabla g$ are (globally) Lipschitz continuous on $\dom(P)$. Consequently, \eqref{qcqp-b} can be suitably solved by Algorithm \ref{PPA} and FPAL \cite{lu2018iteration}. It shall be mentioned that FPAL \cite{lu2018iteration} is only applicable to \eqref{conic-p} with $\dom(P)$ being compact. Our aim is to find a $10^{-2}$-KKT solution of \eqref{qcqp-b} by Algorithm \ref{PPA} and FPAL, and compare their performance. Due to this, we terminate them once a $10^{-2}$-KKT solution is found. Besides, for both methods, we choose zero vector as the initial point and set their parameters as follows.
\begin{itemize}
\item $(\varepsilon,M,\delta,\rho_0,\alpha_0,\eta_0,\zeta,\sigma)=(10^{-2},500,0.9,10,1,0.1,2,0.4)$ for Algorithm \ref{PPA};
\item $\epsilon=10^{-2}$, $\rho_k=\rho_0\zeta^k$, $\eta_k=\eta_0\sigma^k$ with $(\rho_0,\eta_0,\zeta,\sigma)=(10,0.1,2,0.4)$ for FPAL \cite{lu2018iteration}.
\end{itemize}

The computational results of Algorithm \ref{PPA} and FPAL for the instances generated above are presented in Table \ref{t-qcqp-b}. In detail, the value of $n$ is listed in the first column. For each $n$, the average number of gradient evaluations and the average CPU time (in seconds) of Algorithm \ref{PPA} and FPAL over $10$ random instances are given in the rest of the columns. One can observe that our method, namely Algorithm \ref{PPA}, significantly outperforms FPAL in terms of average number of gradient evaluations and average CPU time. This phenomenon is not surprising because Algorithm \ref{PPA} uses a local Lipschitz constant of the gradient of the smooth component of the AL functions, while FPAL uses its global Lipschitz constant that can be excessively conservative.

\begin{table}[H]
\centering
\begin{tabular}{c||ll||ll}
\hline
&\multicolumn{2}{c||}{Gradient evaluations}&\multicolumn{2}{c}{CPU time (seconds)}\\
$n$ &Algorithm \ref{PPA}&FPAL&Algorithm \ref{PPA}&FPAL\\\hline
100&$4.97\times10^3$&$3.96\times10^3$&0.20&0.21\\
200&$4.57\times10^3$&$5.37\times10^3$&6.35&12.07\\
300&$4.47\times10^3$&$6.23\times10^3$&12.53&27.88\\
400&$4.18\times10^3$&$8.00\times10^3$&33.12&93.04\\
500&$4.18\times10^3$&$9.75\times10^3$&22.65&248.28\\
600&$4.18\times10^3$&$1.32\times10^4$&58.65&617.72\\
700&$4.08\times10^3$&$1.22\times10^4$&122.52&889.71\\
800&$4.08\times10^3$&$1.57\times10^4$&186.23&1551.13\\
900&$4.08\times10^3$&$1.94\times10^4$&305.97&2737.96\\
1000&$4.08\times10^3$&$2.30\times10^4$&429.38&4398.44\\
\hline
\end{tabular}
\caption{Numerical results for problem \eqref{qcqp-b}}\label{t-qcqp-b}
\end{table}

\subsection{Quadratically constrained quadratic programming}
In this subsection we consider the quadratically constrained quadratic programming (QCQP)
\begin{align}\label{qcqp}
\begin{aligned}
\min&\  \frac{1}{2}x^TAx+b^Tx\\
\mbox{s.t.}&\ \frac{1}{2}x^TB_ix+c_i^Tx+d_i\leq0,\quad i=1,\dots,m,
\end{aligned}
\end{align}
where $A,B_1,\dots,B_m\in\rr^{n\times n}$ are positive semidefinite matrices, $b, c_1,\dots,c_m\in\rr^{n}$, and $d_1,\dots,d_m\in\rr$.

For each dimension $n$, we set $m=\lceil 0.05n\rceil$ and randomly generate $10$ instances of problem \eqref{qcqp}. In particular, we first generate $x^*\in\rr^n$ with all the entries independently chosen from the standard normal distribution, and $\lambda^*\in\rr_+^m$ whose entries are first independently chosen from the normal distribution with mean $1$ and standard deviation $1$ and then projected to $\rr_+$. We then generate $A$ and $B_i,c_i,\ i=1,\dots,m$ in the same manner as described in Subsection \ref{sec:qcqp-b}. We finally  choose $b$ and $d_i,\ i=1,\dots,m$ so that the KKT conditions of \eqref{qcqp} are satisfied at $(x^*, \lambda^*)$, namely $(x^*, \lambda^*)$ is a KKT point of \eqref{qcqp}.

Notice that \eqref{qcqp} is a special case of \eqref{conic-p} with $f(x)=x^TAx/2+b^Tx$, $P(x)=0$, $g_i(x)=x^TB_ix/2+c_i^Tx+d_i,\ i=1,\dots,m$, and $\mcK=\rr_+^m$. Clearly, $f$ and $P$ are convex, $g$ is $\mcK$-convex, $\nabla f$ and $\nabla g$ are Lipschitz continuous, while $\dom(P)=\rr^n$ is unbounded. As a result, \eqref{qcqp} can be suitably solved by Algorithm \ref{PPA} but not FPAL \cite{lu2018iteration}, since the latter method is only applicable to \eqref{conic-p} with $\dom(P)$ being compact. On the other hand, it is not hard to observe that problem \eqref{qcqp} and its dual can be solved as the monotone inclusion problem
\begin{align}\label{conic-inclusion}
0\in F(x,\lambda)+B(x,\lambda),
\end{align}
where 
\begin{align*}
F(x,\lambda)=\begin{pmatrix}
\nabla f(x)+\nabla g(x)\lambda\\
-g(x)\end{pmatrix},\quad
B(x,\lambda)=\begin{pmatrix}
0\\
\cN_{\rr_+^m}(\lambda)
\end{pmatrix}.
\end{align*}
One can also observe that $F$ is monotone and locally Lipschitz continuous on $\cl(\dom B)$ and $B$ is maximal monotone.  As a result,  problem \eqref{conic-inclusion} and hence \eqref{qcqp} can be suitably solved by FRBS \cite{mali20} and MFBS \cite{tseng2000modified}.  Our aim is to find a $10^{-2}$-KKT solution of \eqref{qcqp} by Algorithm \ref{PPA}, FRBS and MFBS, and compare their performance. Due to this, we terminate them once a $10^{-2}$-KKT solution is found. In addition, for all the methods, we choose zero vector as the initial point and set their parameters as follows.
\begin{itemize}
\item $(\varepsilon,M,\delta,\rho_0,\alpha_0,\eta_0,\zeta,\sigma)=(10^{-2},500,0.9,10,1,0.1,2,0.4)$ for Algorithm \ref{PPA};
\item $(\lambda_0,\delta,\sigma)=(0.1,0.5,0.9)$ for FRBS \cite{mali20};
\item $(\sigma,\theta,\beta)=(0.1,0.5,0.9)$ for MFBS \cite{tseng2000modified}.
\end{itemize}

The computational results of Algorithm \ref{PPA}, FRBS and MFBS for the instances generated above are presented in Table \ref{t-qcqp}. In detail, the value of $n$ is listed in the first column. For each $n$, the average number of gradient evaluations and the average CPU time (in seconds) for these methods over $10$ random instances are given in the rest of the columns. One can observe that our method, namely Algorithm \ref{PPA}, significantly outperforms the other two methods in terms of average number of gradient evaluations and average CPU time. 
This phenomenon may not be surprising because our method enjoys a nearly optimal operation complexity while the other two methods lack complexity guarantees.  
\begin{table}[H]
\centering
\begin{tabular}{c||lll||lll}
\hline
&\multicolumn{3}{c||}{Gradient evaluations}&\multicolumn{3}{c}{CPU time (seconds)}\\
$n$ &Algorithm \ref{PPA}&FRBS&MFBS&Algorithm \ref{PPA}&FRBS&MFBS\\\hline
100&$5.02\times10^3$&$1.89\times10^5$&$1.62\times10^5$&0.17&1.80&1.40\\
200&$5.01\times10^3$&$1.38\times10^5$&$1.36\times10^5$&7.16&80.49&77.47\\
300&$4.68\times10^3$&$1.11\times10^5$&$1.02\times10^5$&12.10&147.19&132.95\\
400&$4.28\times10^3$&$9.76\times10^4$&$8.33\times10^4$&32.52&387.91&323.89\\
500&$4.08\times10^3$&$7.12\times10^4$&$6.16\times10^4$&23.18&147.25&125.38\\
600&$4.18\times10^3$&$7.37\times10^4$&$6.07\times10^4$&53.51&427.64&346.31\\
700&$4.08\times10^3$&$6.59\times10^4$&$5.33\times10^4$&101.14&637.33&518.73\\
800&$4.08\times10^3$&$5.69\times10^4$&$4.72\times10^4$&152.60&782.77&629.80\\
900&$4.09\times10^3$&$5.46\times10^4$&$4.54\times10^4$&236.99&1359.02&1138.14\\
1000&$4.08\times10^3$&$4.63\times10^4$&$3.92\times10^4$&384.96&1854.11&1568.45\\
\hline
\end{tabular}
\caption{Numerical results for problem \eqref{qcqp}}\label{t-qcqp}
\end{table}

\section{Proof of the main results} \label{sec:proof}

In this section we provide a proof of our main results presented in Sections \ref{alg-unconstr} and \ref{alg-constr}, which are particularly Theorems \ref{inner}-\ref{thm:constr}.

\subsection{Proof of the main results in Subsection \ref{alg-unconstr1}}
\label{sec:proof1}

In this subsection we first establish several technical lemmas and then use them to prove Theorems \ref{inner} and \ref{t0}.

\begin{lemma}\label{l3}
Suppose that $\alpha_t$, $\beta_t$ and $\gamma_t$ are generated by Algorithm~\ref{alg-acc} for some $t \ge 1$. Then the following statements hold.
\begin{enumerate}[label=(\roman*)]
\item $\sqrt{\mu\gamma_t}\leq\alpha_t\leq1$ and $\alpha_t^2\gamma_{t}^{-1}\leq\alpha_{t-1}^2\gamma_{t-1}^{-1}$.\label{unc-l3}
\item $\beta_t=\mu\gamma_t\alpha_t^{-1}\in[0,1]$.
\end{enumerate}
\end{lemma}

\begin{proof}
(i) We first prove by induction that $\sqrt{\mu\gamma_i}\leq\alpha_i\leq1$ for all $1\leq i\leq t$. Indeed, notice from Algorithm~\ref{alg-acc} that $\sqrt{\mu\gamma_0}\leq\alpha_0\leq1$. Suppose that $\sqrt{\mu\gamma_{i-1}}\leq\alpha_{i-1}\leq1$ for some 
$1\leq i<t$. By this, \eqref{unc-equa}, and $\alpha_i\in(0,1]$, one has
\[
\gamma_{i-1}\alpha_i^2=(1-\alpha_i)\alpha_{i-1}^2\gamma_i+\mu\alpha_i\gamma_i\gamma_{i-1}\geq(1-\alpha_i)\mu\gamma_{i-1}\gamma_i+\mu\alpha_i\gamma_i\gamma_{i-1}=\mu\gamma_i\gamma_{i-1},
\]
which together with $\gamma_{i-1}>0$ yields $\sqrt{\mu\gamma_i}\leq\alpha_i\leq1$. Hence, the induction is completed and $\sqrt{\mu\gamma_t}\leq\alpha_t\leq1$ holds as desired.

We next show that $\alpha_t^2\gamma_{t}^{-1}\leq\alpha_{t-1}^2\gamma_{t-1}^{-1}$. Indeed, by $\sqrt{\mu\gamma_{t-1}}\leq\alpha_{t-1}$, $\gamma_{t-1}, \gamma_t>0$, and \eqref{unc-equa}, one has
\[
\gamma_{t-1}\alpha_t^2=(1-\alpha_t)\alpha_{t-1}^2\gamma_t+\mu\alpha_t\gamma_t\gamma_{t-1}\leq(1-\alpha_t)\alpha_{t-1}^2\gamma_t+\alpha_t\gamma_t\alpha_{t-1}^2=\gamma_t\alpha_{t-1}^2,
\]
which implies that the conclusion holds.

(ii) Notice from Algorithm~\ref{alg-acc} that $\beta_t=\mu\gamma_t\alpha_t^{-1}$. By this and statement (i), one has 
\[
0\leq\beta_t=\mu\gamma_t\alpha_t^{-1}\leq\sqrt{\mu\gamma_t}\leq 1.
\]
\end{proof}

\begin{lemma}
Suppose that $x^{t+1}$, $y^t$ and $z^{t+1}$ are generated by Algorithm~\ref{alg-acc} for some $t \ge 1$. Then for all $x\in \dom( P)$ and $ P'(z^{t+1})\in\partial  P(z^{t+1})$, we have
\begin{equation}\label{unc-l1}
\gamma_t\langle  P'(z^{t+1}),z^{t+1}-x\rangle\leq\gamma_t\langle\nabla f(y^t),x-z^{t+1}\rangle+\frac{1}{2}\alpha_t\beta_t\|x-y^t\|^2+\frac{1}{2}\alpha_t(1-\beta_t)\|x-z^t\|^2-\frac{1}{2}\alpha_t\|x-z^{t+1}\|^2+R_t,
\end{equation}
where
\begin{equation}\label{def-Rt}
R_t=\frac{1}{2}\mu\gamma_t(\alpha_t^{-1}-1)\|x^t-y^t\|^2-\frac{1}{2\alpha_t}\|x^{t+1}-y^t\|^2.
\end{equation}
\end{lemma}

\begin{proof}
 By the optimality condition of \eqref{unc-prox}, one has
\begin{equation*}
\langle\gamma_t\nabla f(y^t)+\gamma_t P'(z^{t+1})+\alpha_t(z^{t+1}-\beta_ty^t-(1-\beta_t)z^t),x-z^{t+1}\rangle\geq0
\end{equation*}
for all $x\in \dom( P)$ and $ P'(z^{t+1})\in\partial  P(z^{t+1})$. It follows from this relation that
\begin{align}
\gamma_t\langle  P'(z^{t+1}),z^{t+1}-x\rangle\leq\ &\gamma_t\langle\nabla f(y^t),x-z^{t+1}\rangle+\alpha_t\langle z^{t+1}-\beta_t y^t-(1-\beta_t)z^t,x-z^{t+1}\rangle \notag\\
=\ &\gamma_t\langle\nabla f(y^t),x-z^{t+1}\rangle+\alpha_t\beta_t\langle z^{t+1}-y^t,x-z^{t+1}\rangle
+\alpha_t(1-\beta_t)\langle z^{t+1}-z^t,x-z^{t+1}\rangle \notag \\
=\ &\gamma_t\langle\nabla f(y^t),x-z^{t+1}\rangle+\frac{1}{2}\alpha_t\beta_t\left(\|x-y^t\|^2-\|x-z^{t+1}\|^2-\|y^t-z^{t+1}\|^2\right) \notag \\
&+\frac{1}{2}\alpha_t(1-\beta_t)\left(\|x-z^t\|^2-\|x-z^{t+1}\|^2-\|z^t-z^{t+1}\|^2\right) \notag \\
=\ &\gamma_t\langle\nabla f(y^t),x-z^{t+1}\rangle+\frac{1}{2}\alpha_t\beta_t\|x-y^t\|^2+\frac{1}{2}\alpha_t(1-\beta_t)\|x-z^t\|^2 \notag \\
&-\frac{1}{2}\alpha_t\|x-z^{t+1}\|^2+Q_t, \label{Psi-bdd}
\end{align}
where
\begin{equation}\label{def-Qt}
Q_t = -\frac{1}{2}\alpha_t\beta_t\|y^t-z^{t+1}\|^2-\frac{1}{2}\alpha_t(1-\beta_t)\|z^t-z^{t+1}\|^2.
\end{equation}

We next show that $Q_t \le R_t$. Indeed, it follows from \eqref{unc-iter} that 
\beq \label{l2-e2}
x^t-y^t=\alpha_t(1-\alpha_t)^{-1}(1-\beta_t)(y^t-z^t),
\eeq
which together with \eqref{unc-updx} implies that
\begin{align}
x^{t+1}-y^t=\ &(1-\alpha_t)x^t+\alpha_tz^{t+1}-y^t = (1-\alpha_t)(x^t-y^t)+\alpha_tz^{t+1}-\alpha_t y^t\notag\\
\overset{\eqref{l2-e2}}{=}\ &\alpha_t(1-\beta_t)(y^t-z^t)+\alpha_tz^{t+1}-\alpha_t y^t
= \alpha_t\left(z^{t+1}-\beta_t y^t - (1-\beta_t)z^t\right). \label{l2-e3}
\end{align}
Using this relation, $\beta_t\in[0,1]$, and the convexity of $\|\cdot\|^2$, we obtain 
\[
\alpha_t^{-2}\|x^{t+1}-y^t\|^2 \overset{\eqref{l2-e3}}{=} \|z^{t+1}-\beta_t y^t - (1-\beta_t)z^t\|^2 \le \beta_t   \|z^{t+1}-y^t\|^2 + (1-\beta_t)\|z^{t+1}-z^t\|^2.
\]
By this, \eqref{def-Rt}, \eqref{def-Qt}, and $\alpha_t\in(0,1]$, one has
\begin{align*}
2\alpha_t^{-1}(Q_t-R_t)=\ &-\beta_t\|y^t-z^{t+1}\|^2-(1-\beta_t)\|z^t-z^{t+1}\|^2+\alpha_t^{-2}\|x^{t+1}-y^t\|^2-\mu\gamma_t\alpha_t^{-2}(1-\alpha_t)\|x^t-y^t\|^2\\
\le \ &-\beta_t\|y^t-z^{t+1}\|^2-(1-\beta_t)\|z^t-z^{t+1}\|^2+\alpha_t^{-2}\|x^{t+1}-y^t\|^2 \leq0,
\end{align*}
which along with $\alpha_t>0$ implies that $Q_t \le R_t$.

The conclusion of this Lemma directly follows from \eqref{Psi-bdd} and $Q_t \le R_t$.
\end{proof}

\begin{lemma}\label{l4}
Suppose that $x^{t+1}$, $y^t$ and $z^{t+1}$ are generated by Algorithm~\ref{alg-acc} for some $t \ge 1$. Then for any $x\in \dom( P)$, we have
\begin{equation}\label{unc-l4}
F(x^{t+1})-F(x)+\frac{\alpha_t^2}{2\gamma_t}\|x-z^{t+1}\|^2\leq\prod_{i=1}^t(1-\alpha_i)\left(F(x^1)-F(x)+\frac{\alpha^2_{0}}{2\gamma_{0}}\|x-x^1\|^2\right).
\end{equation}
\end{lemma}

\begin{proof}
 By \eqref{unc-updx}, \eqref{unc-l1}, and the convexity of $ P$, one has that for all $ P'(z^{t+1})\in\partial  P(z^{t+1})$, 
\begin{align}
\gamma_t \alpha_t^{-1} P(x^{t+1})\leq & \ \gamma_t \alpha_t^{-1}\left((1-\alpha_t) P(x^{t})+\alpha_t  P(z^{t+1})\right) = \gamma_t(\alpha_t^{-1}-1) P(x^t)+\gamma_t P(z^{t+1})\notag\\
\leq & \ \gamma_t(\alpha_t^{-1}-1) P(x^t)+\gamma_t P(x)+\gamma_t\langle  P'(z^{t+1}),z^{t+1}-x\rangle\notag\\
\overset{\eqref{unc-l1}}{\leq} & \ \gamma_t(\alpha_t^{-1}-1) P(x^t)+\gamma_t P(x)+\gamma_t\langle\nabla f(y^t),x-z^{t+1}\rangle\notag\\
\label{l4-e1}
& \ +\frac{1}{2}\alpha_t(1-\beta_t)\|x-z^t\|^2
-\frac{1}{2}\alpha_t\|x-z^{t+1}\|^2+\frac{1}{2}\alpha_t\beta_t\|x-y^t\|^2+R_t.
\end{align}
By \eqref{convex-1}, \eqref{unc-updx}, $\alpha_t \in (0,1]$, and $\gamma_t>0$, one has that for all $x\in \dom( P)$, 
\begin{align}
&\gamma_t \alpha_t^{-1}f(y^t)+\gamma_t \alpha_t^{-1}\langle\nabla f(y^t),x^{t+1}-y^t\rangle+\gamma_t\langle\nabla f(y^t),x-z^{t+1}\rangle\notag\\
\overset{\eqref{unc-updx}}{=}\ &\gamma_t \alpha_t^{-1}f(y^t)+\gamma_t \alpha_t^{-1}\langle\nabla f(y^t),(1-\alpha_t)x^t+\alpha_tz^{t+1}-y^t\rangle+\gamma_t\langle\nabla f(y^t),x-z^{t+1}\rangle\notag\\
=\ &\gamma_t\alpha_t^{-1}f(y^t)+\gamma_t(\alpha_t^{-1}-1)\langle\nabla f(y^t),x^t-y^t\rangle+\gamma_t\langle\nabla f(y^t),x-y^t\rangle\notag\\
=\ &\gamma_t(\alpha_t^{-1}-1)\left(f(y^t)+\langle\nabla f(y^t),x^t-y^t\rangle\right)+\gamma_t\left(f(y^t)+\langle\nabla f(y^t),x-y^t\rangle\right)\notag\\
\overset{\eqref{convex-1}}{\leq}\ &\gamma_t(\alpha_t^{-1}-1)\left(f(x^t)-\frac{1}{2}\mu\|x^t-y^t\|^2\right)+\gamma_t\left(f(x)-\frac{1}{2}\mu\|x-y^t\|^2\right) \notag\\
=\ &\gamma_t(\alpha_t^{-1}-1)f(x^t)+\gamma_tf(x)-\frac{1}{2}\mu\gamma_t(\alpha_t^{-1}-1)\|x^t-y^t\|^2-\frac{1}{2}\mu\gamma_t\|x-y^t\|^2. \label{l4-e2}
\end{align}
Using \eqref{unc-line}, \eqref{l4-e1} and \eqref{l4-e2}, we have
\begin{align}
\gamma_t \alpha_t^{-1}F(x^{t+1})\overset{\eqref{unc-line}}{\leq}\ &\gamma_t \alpha_t^{-1}f(y^t)+\gamma_t \alpha_t^{-1}\langle\nabla f(y^t),x^{t+1}-y^t\rangle+\frac{1}{2\alpha_t}\|x^{t+1}-y^t\|^2+\gamma_t \alpha_t^{-1} P(x^{t+1})\notag\\
\overset{\eqref{l4-e1}}{\leq}\ &\gamma_t \alpha_t^{-1}f(y^t)+\gamma_t \alpha_t^{-1}\langle\nabla f(y^t),x^{t+1}-y^t\rangle+\gamma_t\langle\nabla f(y^t),x-z^{t+1}\rangle
+\frac{1}{2\alpha_t}\|x^{t+1}-y^t\|^2\notag\\
&+\gamma_t(\alpha_t^{-1}-1) P(x^t)+\gamma_t P(x)+\frac{1}{2}\alpha_t(1-\beta_t)\|x-z^t\|^2-\frac{1}{2}\alpha_t\|x-z^{t+1}\|^2 \notag\\
&+\frac{1}{2}\alpha_t\beta_t\|x-y^t\|^2+R_t\notag\\
\overset{\eqref{l4-e2}}{\leq}\ &\gamma_t(\alpha_t^{-1}-1)F(x^t)+\gamma_tF(x)+\frac{1}{2}(\alpha_t\beta_t-\mu\gamma_t)\|x-y^t\|^2-\frac{1}{2}\mu\gamma_t(\alpha_t^{-1}-1)\|x^t-y^t\|^2\notag\\
&+\frac{1}{2\alpha_t}\|x^{t+1}-y^t\|^2+\frac{1}{2}\alpha_t(1-\beta_t)\|x-z^t\|^2-\frac{1}{2}\alpha_t\|x-z^{t+1}\|^2+R_t\notag\\
=\ &\gamma_t(\alpha_t^{-1}-1)F(x^t)+\gamma_tF(x)+\frac{1}{2}\alpha_t(1-\beta_t)\|x-z^t\|^2-\frac{1}{2}\alpha_t\|x-z^{t+1}\|^2, \label{F-ineq}
\end{align}
where the equality follows from \eqref{def-Rt} and $\beta_t=\mu\gamma_t\alpha_t^{-1}$. In addition, it follows from \eqref{unc-equa} and $\beta_t=\mu\gamma_t\alpha_t^{-1}$ that
\beq \label{alpha-eqn}
\gamma_{t-1}\alpha_t^2(1-\beta_t)=\gamma_{t-1}\alpha_t^2 - \gamma_{t-1}\alpha_t^2\beta_t 
=\gamma_{t-1}\alpha_t^2 - \mu\alpha_t\gamma_t\gamma_{t-1} \overset{\eqref{unc-equa}}{=} (1-\alpha_t)\alpha_{t-1}^2\gamma_t.
\eeq
In view of \eqref{F-ineq} and \eqref{alpha-eqn}, one has
\begin{align*}
F(x^{t+1})-F(x)+\frac{\alpha_t^2}{2\gamma_t}\|x-z^{t+1}\|^2\overset{\eqref{F-ineq}}{\leq}\ &(1-\alpha_t)\left(F(x^t)-F(x)\right)+\frac{\alpha_{t}^2(1-\beta_t)}{2\gamma_{t}}\|x-z^t\|^2\\
\overset{\eqref{alpha-eqn}}{=}\ &(1-\alpha_t)\left(F(x^t)-F(x)+\frac{\alpha_{t-1}^2}{2\gamma_{t-1}}\|x-z^t\|^2\right).
\end{align*}
The conclusion of this lemma immediately follows from the above inequality and $z^1=x^1$.
\end{proof}

Suppose that $x^t$ and $z^t$ are generated by Algorithm~\ref{alg-acc} for some $t \ge 1$. For any $0<\gamma\leq\gamma_0$, we define
\begin{align}
& y^t(\gamma)=\left((1-\alpha(\gamma))x^t+\alpha(\gamma)(1-\beta(\gamma))z^t\right)/\left(1-\alpha(\gamma) \beta(\gamma)\right), \label{y-gamma} \\
\label{unc-proxt}
&z^{t+1}(\gamma)=\argmin_{x}\left\{\gamma\langle\nabla f(y^t(\gamma)),x\rangle+\gamma  P(x)+\frac{\alpha(\gamma)}{2}\|x-\beta(\gamma) y^t(\gamma)-(1-\beta(\gamma))z^t\|^2\right\},\\
\label{unc-updxt}
&x^{t+1}(\gamma)=(1-\alpha(\gamma))x^t+\alpha(\gamma) z^{t+1}(\gamma),
\end{align}
where $\beta(\gamma)=\mu\gamma\alpha(\gamma)^{-1}$ and $\alpha(\gamma)\in(0,1]$ satisfies
\begin{equation} \label{alpha-lambda}
\gamma_{t-1}\alpha(\gamma)^2=(1-\alpha(\gamma))\alpha_{t-1}^2\gamma+\mu\gamma\gamma_{t-1}\alpha(\gamma).
\end{equation}

\begin{lemma}\label{l5}
Let $\cS$ and $\chS$ be defined in \eqref{def-S} and \eqref{def-hS}. Suppose that $x^t,z^t\in\cS$, and $y^t(\gamma)$ and $x^{t+1}(\gamma)$ are defined in \eqref{y-gamma} and \eqref{unc-updxt} for some $t \ge 1$. Then $y^t(\gamma)\in\cS$ and $x^{t+1}(\gamma)\in\chS$ for all $0<\gamma\leq\gamma_0$.
\end{lemma}

\begin{proof}
Fix any $0<\gamma\leq\gamma_0$. By the optimality condition of problems \eqref{unc-prob} and \eqref{unc-proxt},  one has
\begin{eqnarray*}
&\langle\gamma\nabla f(y^t(\gamma))+\gamma  P'(z^{t+1}(\gamma))+\alpha(\gamma)(z^{t+1}(\gamma)-\beta(\gamma) y^t(\gamma)-(1-\beta(\gamma))z^t),x^*-z^{t+1}(\gamma)\rangle\geq0,\\
&\langle\gamma\nabla f(x^*)+\gamma  P'(x^*),z^{t+1}(\gamma)-x^*\rangle\geq0,
\end{eqnarray*}
where $ P'(z^{t+1}(\gamma))\in\partial  P(z^{t+1}(\gamma))$ and $ P'(x^*)\in\partial  P(x^*)$. Letting $w=\beta(\gamma) y^t(\gamma)+(1-\beta(\gamma))z^t$ and using the above two inequalities and the convexity of $ P$, we obtain  
\[
\langle\alpha(\gamma)(z^{t+1}(\gamma)-w)+\gamma(\nabla f(y^t(\gamma))-\nabla f(x^*)),x^*-z^{t+1}(\gamma)\rangle\geq\gamma\langle  P'(z^{t+1}(\gamma))- P'(x^*),z^{t+1}(\gamma)-x^*\rangle\geq0,
\]
which yields 
\begin{align}
\alpha(\gamma)\|z^{t+1}(\gamma)-x^*\|^2 &\leq\langle\alpha(\gamma)(x^*-w)+\gamma(\nabla f(y^t(\gamma))-\nabla f(x^*)),x^*-z^{t+1}(\gamma)\rangle \notag \\
& \leq \|\alpha(\gamma)(x^*-w)+\gamma(\nabla f(y^t(\gamma))-\nabla f(x^*))\| \|z^{t+1}(\gamma)-x^*\|. \label{dist-zxs}
\end{align}
In addition, recall from Lemma \ref{l3} that $\sqrt{\mu\gamma_{t-1}}\leq\alpha_{t-1}\leq1$. By this, $\alpha(\gamma)\in(0,1]$, \eqref{alpha-lambda}, and a similar argument as in the proof of Lemma \ref{l3}(ii), one can see that $\beta(\gamma)\in [0,1]$. It then follows from this, \eqref{y-gamma}, the expression of $w$, and $x^t,z^t\in\cS$ that $y^t(\gamma), w\in\cS$. By these, 
 $\alpha(\gamma)>0$, \eqref{def-S}, \eqref{dist-zxs}, and Lemma~\ref{F-Lipschitz}, one has
\begin{align*}
\alpha(\gamma)\|z^{t+1}(\gamma)-x^*\|\overset{\eqref{dist-zxs}}{\leq}\ &\|\alpha(\gamma)(w-x^*)+\gamma(\nabla f(y^t(\gamma))-\nabla f(x^*))\| \leq \alpha(\gamma) \|w-x^*\|+\gamma \|\nabla f(y^t(\gamma))-\nabla f(x^*)\|  \\
\leq\ &\alpha(\gamma) \|w-x^*\|+\gamma L_\cS\|y^t(\gamma)-x^*\|
\overset{\eqref{def-S}}{\leq} \left(\alpha(\gamma)+\gamma L_\cS\right)\frac{\sqrt{2\gamma_0}r_0}{\alpha_0}.
\end{align*}
Using this, \eqref{def-S}, \eqref{unc-updxt}, $\alpha(\gamma)\in(0,1]$, $x^t\in\cS$, and $\gamma\leq\gamma_0$, we obtain that
\begin{align*}
\|x^{t+1}(\gamma)-x^*\|\overset{\eqref{unc-updxt}}{\leq}\ &(1-\alpha(\gamma))\|x^t-x^*\|+\alpha(\gamma)\|z^{t+1}(\gamma)-x^*\|\\
\overset{\eqref{def-S}}{\leq}\ &(1-\alpha(\gamma))\frac{\sqrt{2\gamma_0}r_0}{\alpha_0}+(\alpha(\gamma)+\gamma L_\cS)\frac{\sqrt{2\gamma_0}r_0}{\alpha_0}\\
\leq\ &(1+\gamma_0L_\cS)\frac{\sqrt{2\gamma_0}r_0}{\alpha_0}.
\end{align*}
It then follows from the last relation and \eqref{def-hS} that $x^{t+1}(\gamma)\in\chS$.
\end{proof} 

For the convenience of our subsequent analysis, we define
\begin{equation}\label{def-lambda}
\lambda_0=1, \quad \lambda_t=\prod_{i=1}^{t}(1-\alpha_i).
\end{equation}

\begin{lemma}\label{l6}
Let $\cS$ and $N$ be defined in \eqref{def-S} and \eqref{def-N}. Suppose that $x^t,z^t\in\cS$ for some $t \ge 1$. Then $x^{t+1}$, $y^t$ and $z^{t+1}$ are successfully generated by Algorithm~\ref{alg-acc} at iteration $t$ with $n_t\leq N$, and moreover, $x^{t+1}, y^t, z^{t+1} \in \cS$.
\end{lemma}

\begin{proof}
Let $\gamma=\gamma_0\delta^{N}$ and $y^t(\gamma)$ and $x^{t+1}(\gamma)$ be defined in \eqref{y-gamma} and \eqref{unc-updxt}. By $\delta\in(0,1)$ and \eqref{def-N}, one can observe that $0<\gamma\leq\gamma_0$ and $\gamma\leq L_\chS^{-1}$.  Using these, $x^t,z^t\in\cS$, and Lemma~\ref{l5}, we see that $x^{t+1}(\gamma)\in\chS$ and $y^t(\gamma)\in\cS\subseteq\chS$, where $\chS$ is defined in \eqref{def-hS}. It then follows from $\gamma\leq L_\chS^{-1}$ and Lemma \ref{F-Lipschitz}(ii) that 
\begin{equation*}
2\gamma \left(f(x^{t+1}(\gamma))-f(y^t(\gamma))-\langle\nabla f(y^t(\gamma)),x^{t+1}(\gamma)-y^t(\gamma)\rangle\right)\leq\gamma L_\chS\|x^{t+1}(\gamma)-y^t(\gamma)\|^2\leq\|x^{t+1}(\gamma)-y^t(\gamma)\|^2.
\end{equation*}
This together with the definition of $n_t$ in Algorithm~\ref{alg-acc} implies that $n_t \le N$. It then follows that $x^{t+1}$, $y^t$ and $z^{t+1}$ are successfully generated by Algorithm~\ref{alg-acc}.

Since $x^t,z^t\in\cS$ and $y^t=y^t(\gamma_t)$ for some $0<\gamma_t\leq\gamma_0$, it follows from Lemma~\ref{l5} that $y^t\in\cS$. We next show that $x^{t+1}, z^{t+1} \in \cS$.
Indeed, by \eqref{unc-equa} and \eqref{def-lambda}, one has 
\begin{equation*}
\lambda_t\overset{\eqref{def-lambda}}{=}(1-\alpha_t)\lambda_{t-1}\overset{\eqref{unc-equa}}{=}\frac{\gamma_{t-1}\alpha_t^2-\mu\alpha_t\gamma_t\gamma_{t-1}}{\alpha_{t-1}^2\gamma_t}\lambda_{t-1}\leq\frac{\gamma_{t-1}\alpha_t^2}{\alpha_{t-1}^2\gamma_t}\lambda_{t-1},
\end{equation*}
which along with $\lambda_0=1$ implies that $\gamma_t\lambda_t/\alpha_t^2 \le  \gamma_0/\alpha_0^2$. 
Using this, \eqref{unc-l4} and \eqref{def-lambda}, we obtain that
\begin{align*}
\|z^{t+1}-x^*\|^2\leq\ &\frac{2\gamma_t}{\alpha_t^2}\left(F(x^{t+1})-F(x^*)+\frac{\alpha_t^2}{2\gamma_t}\|z^{t+1}-x^*\|^2\right)\\
\leq\ &\frac{2\gamma_t\lambda_t}{\alpha_t^2}\left(F(x^1)-F(x^*)+\frac{\alpha^2_{0}}{2\gamma_{0}}\|z^1-x^*\|^2\right)\\
\leq\ &\frac{2\gamma_0}{\alpha_0^2}\left(F(x^1)-F(x^*)+\frac{\alpha^2_{0}}{2\gamma_{0}}\|z^1-x^*\|^2\right),
\end{align*}
which together with \eqref{def-S} implies that $z^{t+1}\in\cS$. It then follows from this and  \eqref{unc-updx} that $x^{t+1}\in\cS$.
\end{proof}

We are now ready to prove Theorems \ref{inner} and \ref{t0}. 

\begin{proof}[\textbf{Proof of Theorem \ref{inner}}]
We prove this theorem by induction. Indeed, notice from Algorithm~\ref{alg-acc} that $z^1=x^1\in\cS$. It then follows from Lemma~\ref{l6} that $x^2$, $y^1$ and $z^2$ are successfully generated with $n_1\leq N$ and $x^2, y^1, z^2 \in \cS$. Now, suppose that $x^t$, $y^{t-1}$ and $z^t$ are already generated with $n_{t-1}\leq N$ and $x^t, y^{t-1}, z^t \in \cS$. It then follows from Lemma~\ref{l6} that $x^{t+1}$, $y^t$ and $z^{t+1}$ are successfully generated with $n_t\leq N$ and $x^{t+1}, y^t, z^{t+1} \in \cS$. Hence, the induction is complete and the conclusion of this theorem holds.
\end{proof}

\begin{proof}[\textbf{Proof of Theorem \ref{t0}}]
Observe from \eqref{def-lambda} that $\lambda_i = (1-\alpha_i) \lambda_{i-1} <\lambda_{i-1}$ for all $i \ge 1$. In addition, recall from the proof of Lemma \ref{l6} that $\gamma_i\lambda_i/\alpha_i^2 \le  \gamma_0/\alpha_0^2$ for all $i \ge 1$. By these relations, one has
\begin{equation*}
\frac{1}{\sqrt{\lambda_{i}}}-\frac{1}{\sqrt{\lambda_{i-1}}}=\frac{\lambda_{i-1}-\lambda_i}{\sqrt{\lambda_{i-1}\lambda_i}(\sqrt{\lambda_{i-1}}+\sqrt{\lambda_i})}
\geq\frac{\lambda_{i-1}-\lambda_i}{2\lambda_{i-1}\sqrt{\lambda_i}}=\frac{\alpha_i}{2\sqrt{\lambda_i}}\geq\frac{1}{2}\alpha_0\sqrt{\gamma_i/\gamma_0} \qquad \forall i \ge 1.
\end{equation*}
Summing up the above inequalities for $i=1,2,\dots,t$ and using $\lambda_0=1$,  we obtain
\begin{equation}\label{p1-e2}
\frac{1}{\sqrt{\lambda_t}}-1\geq\frac{1}{2}\alpha_0\sum_{i=1}^t\sqrt{\gamma_i/\gamma_0}\quad\Rightarrow\quad \lambda_t\leq4\left(2+\alpha_0\sum_{i=1}^t\sqrt{\gamma_i/\gamma_0}\right)^{-2}.
\end{equation}
Also, observe from \eqref{def-lambda} and Lemma~\ref{l3}(i) that 
\begin{equation}\label{p1-e3}
\lambda_t=\prod_{i=1}^{t}(1-\alpha_i)\leq\prod_{i=1}^{t}\left(1-\sqrt{\mu\gamma_i}\right).
\end{equation}
In addition, recall from Theorem \ref{inner} that $n_i \le N$, which together with \eqref{def-N} implies that $\gamma_i=\gamma_0\delta^{n_i}\geq\min\{\gamma_0,\delta/L_\chS\}$ for all $i\geq 1$. By this, \eqref{p1-e2} and \eqref{p1-e3}, one has 
\begin{equation*}
\lambda_t\leq\min\left\{\left(1-\sqrt{\mu\min\left\{\gamma_0,\delta L_\chS^{-1}\right\}}\ \right)^t,\;4\left(2+t\alpha_0\sqrt{\min\left\{1,\delta \gamma_0^{-1}L_\chS^{-1}\right\}}\  \right)^{-2}\right\} 
\qquad \forall t \ge 1.
\end{equation*}
The conclusion of Theorem \ref{t0} then directly follows from this relation, \eqref{def-lambda} and \eqref{unc-l4} with $x=x^*$.  
\end{proof}

\subsection{Proof of the main results in Subsection \ref{alg-unconstr2}}\label{sec:proof2}

In this subsection we first establish two technical lemmas and then use them to prove Theorem~\ref{res-complexity}.

\begin{lemma}\label{l8}
Let $\gamma_0$, $\delta$ be given in Algorithm \ref{alg-acc-term} and $N$ be defined in \eqref{def-N}. Suppose that $(v, \gamma_0, \delta)$ is the input for Algorithm~\ref{alg-term} for any $v\in\cS$. Then $(\tv,\tg)$ is successfully generated by Algorithm~\ref{alg-term} with $\tn\leq N$, $\tv\in\chS$ 
and $\tg\geq\min\{\gamma_0,\delta/L_\chS\}$.
\end{lemma}

\begin{proof}
For any $0<\gamma\leq\gamma_0$, let
\begin{equation}\label{unc-xplust}
\tv(\gamma)=\argmin_{x}\left\{\gamma\langle\nabla f(v),x-v\rangle+\gamma  P(x)+\frac{1}{2}\|x-v\|^2\right\}.
\end{equation}
By the optimality condition of \eqref{unc-prob} and \eqref{unc-xplust} and a similar argument as for \eqref{dist-zxs}, one has
\[
\|\tv(\gamma)-x^*\|\leq \|v-x^* -\gamma(\nabla f(v)-\nabla f(x^*))\|. 
\]
Using this, $v\in\cS$, \eqref{def-S}, and Lemma \ref{F-Lipschitz}(i), we obtain
\begin{align*}
\|\tv(\gamma)-x^*\| \leq \|v-x^*\|+\gamma L_\cS\|v-x^*\| \leq \left(1+\gamma_0 L_\cS\right)\frac{\sqrt{2\gamma_0}r_0}{\alpha_0} \qquad \forall 0<\gamma\leq\gamma_0.
\end{align*}
This along with the definition of $\chS$ in \eqref{def-hS} implies that $\tv(\gamma)\in\chS$ for all $0<\gamma\leq\gamma_0$. Now, let $\gamma=\gamma_0\delta^{N}$. By $\delta\in(0,1)$ and \eqref{def-N}, one can observe that $0<\gamma\leq\gamma_0$ and $\gamma\leq L_\chS^{-1}$. It then follows that $\tv(\gamma)\in\chS$. By these, $v\in\cS\subseteq\chS$ and Lemma \ref{F-Lipschitz}(ii), one has 
\begin{equation*}
2\gamma(f(\tv(\gamma))-f(v)-\langle\nabla f(v),\tv(\gamma)-v\rangle)\leq\gamma L_\chS\|\tv(\gamma)-v\|^2\leq\|\tv(\gamma)-v\|^2.
\end{equation*}
These together with  \eqref{def-N} and the definition of $\tn$ in Algorithm~\ref{alg-term} implies that $(\tv,\tg)$ is successfully generated by Algorithm~\ref{alg-term} with $\tn \le N$, and moreover, 
\[
\gamma_0 \geq \tg= \gamma_0\delta^{\tn}\geq \gamma_0\delta^{N}\geq\min\{\gamma_0,\delta/L_\chS\}, \quad \tv=\tv(\tg)\in\chS.
\] 
\end{proof}

\begin{lemma} \label{lem:res-term}
Suppose that $x^{t+1}$ and $(\tx^{t+1},\tg_{t+1})$ are generated in Algorithm~\ref{alg-acc-term} for some $t \ge 1$. Then we have
\begin{align}
\dist(0,\partial F(\tx^{t+1})) & \leq\|\tg_{t+1}^{-1}(x^{t+1}-\tx^{t+1})+\nabla f(\tx^{t+1})-\nabla f(x^{t+1})\|  \notag \\
& \leq \left(\sqrt{2\max\{\gamma_0^{-1},L_\chS\delta^{-1}\}}+\sqrt{2\gamma_0}L_\chS\right)\sqrt{F(x^{t+1})-F(x^*)}, \label{residual-ineq}
\end{align}
where $L_\chS$ is given in Lemma \ref{F-Lipschitz}, and $\gamma_0$ and $\delta$ are the input parameters of Algorithm~\ref{alg-acc}.
\end{lemma}

\begin{proof}
Notice that $(\tx^{t+1},\tg_{t+1})$ is the output of Algorithm \ref{alg-term} with $(x^{t+1},\gamma_0,\delta)$ as the input. By Lemma~\ref{l8}, one has that $\tx^{t+1}\in\chS$ and $\gamma_0\geq\tg_{t+1}\geq\min\{\gamma_0,\delta/L_\chS\}$. Also, it follows from \eqref{unc-xplus} and \eqref{unc-check1} with 
$v=x^{t+1}$, $\tv=\tx^{t+1}$ and $\tg=\tg_{t+1}$ that
\begin{align}
&\tx^{t+1}=\argmin_{x}\left\{\tg_{t+1}\langle\nabla f(x^{t+1}),x\rangle+\tg_{t+1} P(x)+\frac{1}{2}\|x-x^{t+1}\|^2\right\},  \label{subprob-txt} \\
& 2\tg_{t+1}(f(\tx^{t+1})-f(x^{t+1})-\langle\nabla f(x^{t+1}),\tx^{t+1}-x^{t+1}\rangle)\leq\|\tx^{t+1}-x^{t+1}\|^2. \label{subprob-term-line}
\end{align}
By the optimality condition of \eqref{subprob-txt}, it can be easily shown that
\begin{align}
 & \tg_{t+1}^{-1}(x^{t+1}-\tx^{t+1})+\nabla f(\tx^{t+1})- \nabla f(x^{t+1}) \in \partial F(\tx^{t+1}), \label{F-residual}\\ 
& \tg_{t+1}\langle\nabla f(x^{t+1}),\tx^{t+1}\rangle+\tg_{t+1} P(\tx^{t+1}) \le \tg_{t+1}\langle\nabla f(x^{t+1}),x^{t+1}\rangle+\tg_{t+1} P(x^{t+1}) - \|\tx^{t+1}-x^{t+1}\|^2.  \label{Psi-tx}
\end{align}
By \eqref{subprob-term-line} and \eqref{Psi-tx}, one has
\begin{align*}
 \tg_{t+1}F(\tx^{t+1})\overset{\eqref{subprob-term-line}}{\leq}& \tg_{t+1} P(\tx^{t+1})+ \tg_{t+1} f(x^{t+1})+ \tg_{t+1}\langle\nabla f(x^{t+1}),\tx^{t+1}-x^{t+1}\rangle+\frac{1}{2}\|\tx^{t+1}-x^{t+1}\|^2\\
\overset{\eqref{Psi-tx}}{\leq}& \tg_{t+1}F(x^{t+1})-\frac{1}{2}\|\tx^{t+1}-x^{t+1}\|^2,
\end{align*}
which yields $\|\tx^{t+1}-x^{t+1}\|\le\sqrt{2\tg_{t+1}(F(x^{t+1})-F(\tx^{t+1}))}$.
This together with \eqref{F-residual}, $\tx^{t+1}\in\chS$, $\gamma_0\geq\tg_{t+1}\geq\min\{\gamma_0,\delta/L_\chS\}$, and Lemma \ref{F-Lipschitz}(ii) implies
\begin{align*}
\dist(0,\partial F(\tx^{t+1}))& \leq\|\tg_{t+1}^{-1}(x^{t+1}-\tx^{t+1})+\nabla f(\tx^{t+1})-\nabla f(x^{t+1})\|\leq (\tg_{t+1}^{-1}+L_\chS)\|\tx^{t+1}-x^{t+1}\| \\
& \leq\left(\sqrt{2\tg_{t+1}^{-1}}+\sqrt{2\tg_{t+1}}L_\chS\right)\sqrt{F(x^{t+1})-F(\tx^{t+1})} \\
& \leq\left(\sqrt{2\max\{\gamma_0^{-1},L_\chS\delta^{-1}\}}+\sqrt{2\gamma_0}L_\chS\right)\sqrt{F(x^{t+1})-F(x^*)}.
\end{align*}
\end{proof}

We are now ready to prove Theorem~\ref{res-complexity}.

\begin{proof}[\textbf{Proof of Theorem~\ref{res-complexity}}.]
Suppose for contradiction that Algorithm~\ref{alg-acc-term} does not terminate within $T$ iterations. It then follows that $x^{t+1}$ and $\tx^{t+1}$ must be generated in Algorithm~\ref{alg-acc-term} for some $T-M < t \le T$ with $\mathrm{mod}(t,M)=0$. In addition, observe that \eqref{opt-gap} also holds for Algorithm~\ref{alg-acc-term}. By $t > T-M$, \eqref{opt-gap}, \eqref{def-T} and \eqref{residual-ineq}, one has
\begin{align*}
& \|\tg_{t+1}^{-1}(x^{t+1}-\tx^{t+1})+\nabla f(\tx^{t+1})-\nabla f(x^{t+1})\| \overset{\eqref{residual-ineq}}{\leq} \left(\sqrt{2\max\{\gamma_0^{-1},L_\chS\delta^{-1}\}}+\sqrt{2\gamma_0}L_\chS\right)\sqrt{F(x^{t+1})-F(x^*)} \\
&\overset{\eqref{opt-gap}}{\leq}  r_0 \left(\sqrt{2\max\{\gamma_0^{-1},L_\chS\delta^{-1}\}}+\sqrt{2\gamma_0}L_\chS\right)  \left(1-\sqrt{\mu\min\left\{\gamma_0,\delta L^{-1}_\chS\right\}} \right)^{t/2} \\
&<\    r_0 \left(\sqrt{2\max\{\gamma_0^{-1},L_\chS\delta^{-1}\}}+\sqrt{2\gamma_0}L_\chS\right)  \left(1-\sqrt{\mu\min\left\{\gamma_0,\delta L^{-1}_\chS\right\}}\ \right)^{(T-M)/2}  \overset{\eqref{def-T}}{\leq} \epsilon.
\end{align*}
which implies that Algorithm~\ref{alg-acc-term} terminates at iteration $t$ and leads to a contradiction. Consequently, Algorithm~\ref{alg-acc-term} must terminate at some iteration $t \le T$ and output $\tx^{t+1}$ that satisfies \eqref{unc-check3}. By this and Lemma \ref{lem:res-term}, one can see that $\dist(0,\partial F(\tx^{t+1})) \leq \epsilon$ and hence $\tx^{t+1}$ is an $\epsilon$-residual solution of problem \eqref{unc-prob}.

In addition, one can observe from Algorithm~\ref{alg-acc-term} that (i) evaluations of $\nabla f$ and proximal operator of $P$ are performed in the backtracking line search procedure (see step 2) and Algorithm \ref{alg-term} (see step 4); (ii) the total number of iterations of Algorithm~\ref{alg-acc-term} is at most $T$; (iii) $n_t$  backtracking trials are performed in each iteration $t$ and each of them requires one evaluation of $\nabla f$ and proximal operator of $P$; (iv) the total number of calls of Algorithm \ref{alg-term} in Algorithm~\ref{alg-acc-term} is at most $T/M$ and each call requires at most $N$ evaluations of $\nabla f$ and proximal operator of $P$ (see Algorithm \ref{alg-term} and Lemma \ref{l8}), where $N$ is given in \eqref{def-N}. By this observation and Theorem \ref{inner}, one can see that the total number of evaluations of $\nabla f$ and proximal operator of $P$ performed in Algorithm~\ref{alg-acc-term} is no more than $\bar N$, respectively.  
\end{proof}

\subsection{Proof of the main results in Subsection \ref{alg-unconstr3}}\label{sec:proof3}

In this subsection we first establish several technical lemmas and then use them to prove Theorem~\ref{thm:cvx-unconstr}.


Let $\{x^k\}_{k\in \bbK}$ denote all the iterates generated by Algorithm \ref{PPA-sp}, where $\bbK$ is a subset of consecutive nonnegative integers starting from $0$. We define $\bbK-1 = \{k-1: k \in \bbK\}$. For any $0\leq k \in \bbK-1$, let $f_k$ and $F_k$ be defined in \eqref{fk-sp}. Also, let $x_*^k$ be defined as
\begin{align} \label{sp-sub}
x^k_*=\argmin_x F_k(x).
\end{align}
Recall that $\alpha_0$, $\gamma_0$ and $\{\rho_k\}$ are the input parameters of Algorithm~\ref{PPA-sp}, and $L_{\nabla_f}$ and $\tL_{\nabla_f}$ are the Lipschitz constant of $\nabla f$ on $\mcQ$ and $\chQ$, respectively.
Let 
\begin{align}
L_k &=L_{\nabla_f}+\rho_k^{-1}, \quad \tL_k=\tL_{\nabla_f}+\rho_k^{-1},  \label{conic-L-sp} \\
\br_k &=\sqrt{F_k(x^k)-F_k(x_*^k)+\frac{\alpha^2_0}{2\gamma_0}\|x^k-x_*^k\|^2}, \label{rk-sp} \\
\cS_k&=\left\{x\in \dom( P):\|x-x_*^k\|\leq\alpha_0^{-1}\sqrt{2\gamma_0}\br_k\right\}, \label{def-Sk-sp} \\
\chS_k &=\left\{x\in \dom( P):\|x-x^k_*\|\leq \left(1+\gamma_0L_k\right)\alpha_0^{-1}\sqrt{2\gamma_0}\br_k\right\}. \label{def-hSk-sp} 
\end{align}
Since $L_{\nabla_f}$ and $\tL_{\nabla_f}$ are respectively the Lipschitz constant of $\nabla f$ on $\mcQ$ and $\chQ$, it then follows from \eqref{fk-sp} that $\nabla f_k$ is $L_k$- and $\tL_k$-Lipschitz continuous on $\mcQ$ and $\chQ$, respectively. In addition, 
by the definition of $L$ and $\tL $ in \eqref{def-Q-sp} and the monotonicity of $\{\rho_k\}$, one has
\begin{equation}
L_k =L_{\nabla_f}+\rho_k^{-1}\leq L, \quad \tL_k=\tL_{\nabla_f}+\rho_k^{-1}\leq \tL.  \label{ineq-L-sp}
\end{equation}

\begin{lemma} \label{AL-bound-sp}
Let $x^k_*$ be defined in \eqref{sp-sub}. Then the following statements hold.
\begin{align}
& \|x^k-x_*^k\|^2+\|x_*^k-x^*\|^2  \leq\|x^k-x^*\|^2 \quad \forall 0 \leq k\in \bbK-1, \label{ppa-bnd1-sp} \\
& \|x^k-x^{k-1}\| \leq\|x^0-x^*\|+\sum_{i=0}^{k-1}\rho_i\eta_i, \quad\|x^k-x^*\|\leq\|x^0-x^*\|+\sum_{i=0}^{k-1}\rho_i\eta_i \quad \forall 1 \leq k\in \bbK. \label{ppa-bnd2-sp} 
\end{align}
\end{lemma}

\begin{proof}
One can observe that Algorithm \ref{PPA-sp} is an inexact proximal point algorithm (PPA)  \cite{Rock76a} applied to the monotone inclusion problem $0\in \mcT(x)$, where $\mcT: \rr^n\rightrightarrows\rr^n$ is a maximal monotone set-valued operator defined as
\[
\mcT(x)=\left\{\ba{ll} 
\partial F(x)  & \mbox{if } x\in\dom(P), \\ 
\emptyset  & \mbox{otherwise},
\ea\right.  \qquad \forall x\in\rr^n.
\]
In addition, one can observe from \eqref{subprob-term-sp} and \eqref{sp-sub} that 
$\dist(0, \mcT(x^{k+1})+\rho_k^{-1}(x^{k+1}-x^k)) \leq \eta_k$ and $x^k_*=(I+\rho_k\mcT)^{-1}(x^k)$. It then follows from \cite[Proposition 3]{Rock76a} that 
\beq \label{inexact-ppa-sp}
\|x^{k+1}-(I+\rho_k\mcT)^{-1}(x^k)\| \leq \rho_k\eta_k \quad \forall k\in\bbK-1.
\eeq 
By this, $0\in\mcT(x^*)$, $x^k_*=(I+\rho_k\mcT)^{-1}(x^k)$ and  \cite[Proposition 1]{Rock76a}, one can see that \eqref{ppa-bnd1-sp} holds. In addition, \eqref{ppa-bnd2-sp} follows from  \eqref{inexact-ppa-sp} and \cite[Lemma 3]{lu2018iteration}.
\end{proof}

As a consequence of Lemma~\ref{AL-bound-sp} and the definition of $r_0$ and $\theta$ in \eqref{def1-sp}, one has that 
\begin{equation}\label{AL-ineq-sp}
\|x^0-x_*^0\|\leq r_0, \;\; \|x^k-x^*\|\leq r_0+\theta,\;\;\|x^k-x_*^k\|\leq r_0+\theta, \;\;\|x^k-x^{k-1}\|\leq r_0+\theta, \quad \forall 1 \leq k\in \bbK.
\end{equation}

\begin{lemma}
Let $\tr0$ and $\br^k$ be defined in \eqref{def-tP-sp} and \eqref{rk-sp}.  Then for all $0\leq k\in\bbK-1$, we have
\begin{equation}\label{l11-potential-sp}
\br_k^2 \leq\alpha_0^2\tr0^2/(2\gamma_0).
\end{equation}
\end{lemma}

\begin{proof}
We first prove that \eqref{l11-potential-sp} holds for $k=0$, that is, $\br_0^2 \leq\alpha_0^2\tr0^2/(2\gamma_0)$. By \eqref{unc-prob}, \eqref{fk-sp} and the definition of $x^*$, one has 
\begin{align*}
F_0(x^0_*) =F(x^0_*)+\frac{1}{2\rho_0} \|x^0_*-x^0\|^2\geq F(x^*), \quad F_0(x^0) =F(x^0).
\end{align*}
It then follows from these, \eqref{def-tP-sp}, \eqref{rk-sp}, and \eqref{AL-ineq-sp} that
\begin{align*}
\br_0^2 \overset{\eqref{rk-sp}}{=} F_0(x^0)- F_0(x_*^0)+\frac{\alpha^2_0}{2\gamma_0}\|x^0-x_*^0\|^2
\leq F(x^0)-F(x^*)+\frac{\alpha^2_0r_0^2}{2\gamma_0}
\overset{\eqref{def-tP-sp}}{\leq}\frac{\alpha_0^2\tr0^2}{2\gamma_0}.
\end{align*}

We next show that \eqref{l11-potential-sp} holds for all $1 \leq k\in\bbK-1$. It follows from  \eqref{fk-sp} and \eqref{subprob-term-sp} that there exists $P'(x^k)\in\partial P(x^k)$ such that 
\begin{equation} \label{subg-F1-sp}
F'_{k-1}(x^k)=\nabla f(x^k)+\rho_{k-1}^{-1}(x^k-x^{k-1})+P'(x^k) \in \partial F_{k-1}(x^k), \quad \|F'_{k-1}(x^k)\|\leq\eta_{k-1}.
\end{equation}
Also, we have
\begin{equation*}
\nabla f(x^k)+P'(x^k)\in\partial F_k(x^k),
\end{equation*}
which together with \eqref{subg-F1-sp} yields
\beq \label{subg-F2-sp}
F'_{k-1}(x^k)-\rho_{k-1}^{-1}(x^k-x^{k-1}) \in \partial F_k(x^k).
\eeq
 By the convexity of $F$, $\eta_{k-1} \le \eta_0$, $\rho_{k-1} \ge \rho_0$, \eqref{AL-ineq-sp} and \eqref{subg-F2-sp}, one has
\begin{align*}
F_k(x^k)-F_k(x_*^k)\overset{\eqref{subg-F2-sp}}{\leq}& \ \langle F'_{k-1}(x^k)-\rho_{k-1}^{-1}(x^k-x^{k-1}),x^k-x_*^k\rangle 
\leq  (\|F'_{k-1}(x^k)\|+\rho_{k-1}^{-1}\|x^k-x^{k-1}\|)\|x^k-x_*^k\|\\
\overset{\eqref{AL-ineq-sp}}{\leq}& \ \eta_0(r_0+\theta)+\rho_0^{-1}(r_0+\theta)^2.
\end{align*}
This together with \eqref{rk-sp} and \eqref{AL-ineq-sp} yields
\[
 \br_k^2 =\ F_k(x^k)-F_k(x_*^k)+\frac{\alpha^2_0}{2\gamma_0}\|x^k-x_*^k\|^2
 \leq \eta_0(r_0+\theta)+\rho_0^{-1}(r_0+\theta)^2+\frac{\alpha^2_0(r_0+\theta)^2}{2\gamma_0}.
\]
By this relation and the definition of $\tr0$ in \eqref{def-tP-sp}, one can see that \eqref{l11-potential-sp} holds for all $1 \leq k\in\bbK-1$.
\end{proof}

\begin{lemma}\label{l11-sp}
Let $f_k$,  $L_k$, $\tL_k$, $\cS_k$ and $\chS_k$ be respectively defined in \eqref{fk-sp}, \eqref{conic-L-sp}, \eqref{def-Sk-sp} and \eqref{def-hSk-sp}. Then for all $0\leq k\in\bbK-1$, $\nabla f_k$ is Lipschitz continuous on $\cS_k$ and  $\chS_k$ with Lipschitz constants $L_k$ and $\tL_k$, respectively.
\end{lemma}
\begin{proof}
Let $\mcQ$ and $\chQ$ be defined in \eqref{def-P-sp} and \eqref{def-Q-sp}. We first show that 
$\cS_k\subseteq\mcQ$ and $\chS_k\subseteq\chQ$ for all $0\leq k\in\bbK-1$. To this end, fix any $0\leq k\in\bbK-1$. By \eqref{def-Sk-sp}, \eqref{AL-ineq-sp} and \eqref{l11-potential-sp}, one has that for all $x\in\cS_k$, 
\begin{equation*}
\|x-x^*\|\leq\|x-x_*^k\|+\|x_*^k-x^*\|\overset{\eqref{def-Sk-sp} }{\leq} \alpha_0^{-1}\sqrt{2\gamma_0}\br_k +\|x_*^k-x^*\| \leq \tr0+r_0+\theta,
\end{equation*}
where the last inequality follows from \eqref{AL-ineq-sp} and \eqref{l11-potential-sp}. This together with \eqref{def-P-sp} implies that $\cS_k\subseteq\mcQ$. In addition, using  \eqref{def-hSk-sp}, \eqref{ineq-L-sp}, \eqref{AL-ineq-sp} and \eqref{l11-potential-sp}, we obtain that for all $x\in\chS_k$, 
\begin{align*}
\|x-x^*\| \leq&\  \|x-x_*^k\|+\|x_*^k-x^*\|\overset{\eqref{def-hSk-sp} }{\leq} \left(1+\gamma_0L_k\right)\alpha_0^{-1}\sqrt{2\gamma_0}\br_k + \|x_*^k-x^*\| \\
\leq&\  (1+\gamma_0L_k)\tr0+r_0+\theta \overset{\eqref{ineq-L-sp}}{\leq} (1+\gamma_0L)\tr0+r_0+\theta,
\end{align*}
which along with \eqref{def-Q-sp} implies that $\chS_k\subseteq\chQ$.

Recall that $\nabla f_k$ is $L_k$- and $\tL_k$-Lipschitz continuous on $\mcQ$ and $\chQ$, respectively. The conclusion of this lemma then follows from this fact and the relations $\cS_k\subseteq\mcQ$ and $\chS_k\subseteq\chQ$ for all $0\leq k\in\bbK-1$.
\end{proof}

\begin{lemma}\label{prop-sub-sp}
Let $N_k$ denote the number of evaluations of $\nabla f$ and proximal operator of $P$ performed by Algorithm~\ref{alg-acc-term} at the $k$th outer iteration of  Algorithm~\ref{PPA-sp}. Then for all $0 \le k\in \bbK-1$, it holds that 
\begin{equation}\label{def-Nk-sp}
N_k \leq \tC_1\left(M+1+\frac{\left(\log\frac{\alpha_0^2\tr0^2\left(\sqrt{\max\{\gamma_0^{-2},\gamma_0^{-1}\tL\delta^{-1}\}}+\tL\right)^2}{\eta_k^2}\right)_+}{\sqrt{\rho_k^{-1}\min\left\{\gamma_0,\delta \tL^{-1}\right\}}}\right) ,
\end{equation}
where $M$, $\delta$, $\alpha_0$, $\gamma_0$, $\{\rho_k\}$ and $\{\eta_k\}$ are the input parameters of Algorithm~\ref{PPA-sp}, and $\tr0$, $\tL$ and $\tC_1$ are given in \eqref{def-tP-sp}, \eqref{def-Q-sp} and \eqref{C1-sp}, respectively. 
\end{lemma}

\begin{proof}
Notice that at the $k$th outer iteration of Algorithm~\ref{PPA-sp}, Algorithm~\ref{alg-acc-term} is called to find an $\eta_k$-residual solution $x^{k+1}$ of the problem $\min_x \left\{f_k(x)+P(x)\right\}$ with the inputs $\epsilon \leftarrow\eta_k$, $\mu \leftarrow\rho_k^{-1}$ and $x^1=z^1 \leftarrow x^k$. In view of \eqref{rk-sp}, \eqref{def-Sk-sp}, \eqref{def-hSk-sp}, Lemma \ref{l11-sp} and Theorem \ref{res-complexity}, one can replace $(r_0, \mu, \epsilon, L_\chS)$ in \eqref{def-NN} by $(\br_k, \rho_k^{-1}, \eta_k, \tL_k)$ respectively and obtain that
\begin{align*}
N_k &\leq (1+M^{-1})\left(M+\left\lceil\frac{2\log\frac{\eta_k}{\br_k\left(\sqrt{2\max\{\gamma_0^{-1},\tL_k\delta^{-1}\}}+\sqrt{2\gamma_0}\tL_k\right)}}{\log\left(1-\sqrt{\rho_k^{-1}\min\left\{\gamma_0,\delta \tL_k^{-1}\right\}}\,\right)}\right\rceil_+\right) \left(1+\left\lceil\frac{\log(\gamma_0\tL_k)}{\log(1/\delta)}\right\rceil_+\right) \notag \\
&\leq (1+M^{-1})\left(M+1+\frac{\left(\log\frac{2\br_k^2\left(\sqrt{\max\{\gamma_0^{-1},\tL_k\delta^{-1}\}}+\sqrt{\gamma_0}\tL_k\right)^2}{\eta_k^2}\right)_+}{-\log\left(1-\sqrt{\rho_k^{-1}\min\left\{\gamma_0,\delta \tL_k^{-1}\right\}}\,\right)}\right) \left(1+\left\lceil\frac{\log(\gamma_0\tL_k)}{\log(1/\delta)}\right\rceil_+\right) \notag \\
&\leq (1+M^{-1})\left(M+1+\frac{\left(\log\frac{2\gamma_0\br_k^2\left(\sqrt{\max\{\gamma_0^{-2},\gamma_0^{-1}\tL_k\delta^{-1}\}}+\tL_k\right)^2}{\eta_k^2}\right)_+}{\sqrt{\rho_k^{-1}\min\left\{\gamma_0,\delta \tL_k^{-1}\right\}}}\right) \left(1+\left\lceil\frac{\log(\gamma_0\tL_k)}{\log(1/\delta)}\right\rceil_+\right),
\end{align*}
where the last inequality follows from the fact that $-\log(1-\xi) \ge \xi$ for any $\xi \in(0,1)$.   By  the above inequality, \eqref{C1-sp}, \eqref{ineq-L-sp} and \eqref{l11-potential-sp}, one can see that \eqref{def-Nk-sp} holds.
\end{proof}

We are now ready to prove Theorem~\ref{thm:cvx-unconstr}.

\begin{proof}[\textbf{Proof of Theorem~\ref{thm:cvx-unconstr}}.]
(i) Let $K$ be defined in \eqref{K-sp}. We first show that Algorithm~\ref{PPA-sp} terminates after at most $K+1$ outer iterations. Indeed, suppose for contradiction that it runs for more than $K+1$ outer iterations. It then follows that 
\eqref{ppa-term-sp} does not hold for $k=K$. On the other hand, by \eqref{K-sp}, \eqref{AL-ineq-sp}, $\rho_K=\rho_0\zeta^K$ and  $\eta_K=\eta_0\sigma^K$, one has 
\[
\frac{1}{\rho_K}\|x^{K+1}-x^K\| \leq \frac{r_0+\theta}{\rho_0\zeta^K} \overset{\eqref{K-sp}}{\leq} \frac{\varepsilon}{2}, \qquad \eta_K=\eta_0\sigma^K\overset{\eqref{K-sp}}{\leq}\frac{\varepsilon}{2},
\]
and hence \eqref{ppa-term-sp} holds for $k=K$, which leads to a contradiction. Hence, there exists some $0\leq k \le K$ such that \eqref{ppa-term-sp} holds and Algorithm~\ref{PPA-sp} terminates and outputs $x^{k+1}$. We next show that $x^{k+1}$ is an $\varepsilon$-residual solution of problem \eqref{unc-prob}. Indeed, it follows from \eqref{fk-sp} and \eqref{ppa-term-sp} that 
\begin{align*}
 \dist(0, \partial F(x^{k+1})) & \le \dist(0, \partial F(x^{k+1})+\rho_k^{-1}(x^{k+1}-x^k)) + \rho_k^{-1}
\|x^{k+1}-x^k\| \\
& \overset{\eqref{fk-sp}}{=} \dist(0, \partial F_k(x^{k+1})) + \rho_k^{-1} \|x^{k+1}-x^k\| \leq \eta_k  + \rho_k^{-1}
\|x^{k+1}-x^k\| \overset{\eqref{ppa-term-sp}}{\leq} \varepsilon,
\end{align*}
and hence the output $x^{k+1}$  of Algorithm~\ref{PPA-sp} is an $\varepsilon$-residual solution of problem \eqref{unc-prob}.

(ii) Let $K$ and $\widetilde N$ be defined in \eqref{K-sp} and \eqref{complexity-sp}, and let $N_k$ denote the number of evaluations of $\nabla f$ and proximal operator of $P$ performed by Algorithm~\ref{alg-acc-term} at the $k$th outer iteration of Algorithm~\ref{PPA-sp}. By this and statement (i) of this theorem, one can observe that the total number of evaluations of $\nabla f$ and proximal operator of $P$ performed in Algorithm~\ref{PPA-sp} is no more than $\sum_{k=0}^{|\bbK|-2} N_k $. As a result, to prove statement (ii) of this theorem, it suffices to show that $\sum_{k=0}^{|\bbK|-2} N_k \le \widetilde N$. Indeed, in view of \eqref{C2-sp}, \eqref{K-sp}, \eqref{def-Nk-sp}, $|\bbK|-2 \le K$, $\rho_k=\rho_0\zeta^k$ and $\eta_k=\eta_0\sigma^k$, one has
\begin{align*}
\sum_{k=0}^{|\bbK|-2}N_k & \leq \tC_1 \sum_{k=0}^{K} \left(M+1+\frac{\sqrt{\rho_k}\left(\log\frac{\alpha_0^2\tr0^2\left(\sqrt{\max\{\gamma_0^{-2},\gamma_0^{-1}\tL\delta^{-1}\}}+\tL\right)^2}{\eta_k^2}\right)_+}{\min\left\{\sqrt{\gamma_0},\sqrt{\delta \tL^{-1}}\right\}}\right) \\
& = \tC_1\sum_{k=0}^{K}\left(M+1+\frac{\sqrt{\rho_0}\sqrt{\zeta}^k\left(-2k\log\sigma+\log\frac{\alpha_0^2\tr0^2\left(\sqrt{\max\{\gamma_0^{-2},\gamma_0^{-1}\tL\delta^{-1}\}}+\tL\right)^2}{\eta_0^2}\right)_+}{\min\left\{\sqrt{\gamma_0},\sqrt{\delta \tL^{-1}}\right\}}\right)\\
& \leq \tC_1\left((M+1)(K+1)+\frac{\sqrt{\rho_0}\sqrt{\zeta}^{K+1}\left(-2K\log\sigma+\log\frac{\alpha_0^2\tr0^2\left(\sqrt{\max\{\gamma_0^{-2},\gamma_0^{-1}\tL\delta^{-1}\}}+\tL\right)^2}{\eta_0^2}\right)_+}{(\sqrt{\zeta}-1)\min\left\{\sqrt{\gamma_0},\sqrt{\delta \tL^{-1}}\right\}}\right)  \leq\widetilde N,
\end{align*}
where the first inequality follows from \eqref{def-Nk-sp}, the second inequality is due to 
$\sum^K_{k=0}\sqrt{\zeta}^k \le \sqrt{\zeta}^{K+1}/(\sqrt{\zeta}-1)$ and
$\sum^K_{k=0}k \sqrt{\zeta}^k \le K\sqrt{\zeta}^{K+1}/(\sqrt{\zeta}-1)$, and 
the last inequality follows from \eqref{C2-sp}, \eqref{K-sp} and \eqref{complexity-sp}.
\end{proof}

\subsection{Proof of the main results in Section \ref{alg-constr}}\label{sec:proof4}

In this subsection we first establish several technical lemmas and then use them to prove Theorem~\ref{thm:constr}.


Let $\{(x^k,\lambda^k)\}_{k\in \bbK}$ denote all the iterates generated by Algorithm \ref{PPA}, where $\bbK$ is a subset of consecutive nonnegative integers starting from $0$. We define $\bbK-1 = \{k-1: k \in \bbK\}$. For any $0 \leq k \in \bbK-1$, let $f_k$ and $F_k$ be defined in \eqref{fk}. In addition, let $(x_*^k,\lambda_*^k)$ be defined as
\begin{align} \label{ppa-sub}
x^k_*=\argmin_x F_k(x), \qquad \lambda^k_* =\Pi_{\mcK^*}\left(\lambda^k+\rho_k g(x^k_*)\right).
\end{align}

Recall that $\alpha_0$, $\{\rho_k\}$ and $\{\eta_k\}$ are the input parameters of Algorithm~\ref{PPA}, $\mcQ$, $B$, $C$, $\chQ$, $\tB$ and $\wC$ are respectively given in \eqref{def-P}, \eqref{def2}, \eqref{def-Q} and \eqref{def3}, and $L_{\nabla g}$ and $\tL_{\nabla g}$ are the Lipschitz constant of $\nabla g$ on $\mcQ$ and $\chQ$, respectively.  Let
\begin{align}
L_k &=C\rho_k+B+L_{\nabla g}\sum_{i=0}^{k-1}\rho_i\eta_i+\rho_k^{-1}, \quad \tL_k=\wC\rho_k+\tB+\tL_{\nabla g}\sum_{i=0}^{k-1}\rho_i\eta_i+\rho_k^{-1},  \label{conic-L} \\
\br_k &=\sqrt{F_k(x^k)-F_k(x_*^k)+\frac{1}{2}\rho_k\alpha^2_0\|x^k-x_*^k\|^2}, \label{rk} \\
\cS_k&=\left\{x\in \dom( P):\|x-x_*^k\|\leq\alpha_0^{-1}\sqrt{2\rho_k^{-1}}\br_k\right\}, \label{def-Sk} \\
\chS_k &=\left\{x\in \dom( P):\|x-x^k_*\|\leq \left(1+L_k\rho_k^{-1}\right)\alpha_0^{-1}\sqrt{2\rho_k^{-1}}\br_k\right\}. \label{def-hSk} 
\end{align}

The following lemma states some properties of the function $f_k$, whose proof is similar to that of \cite[Lemma 7]{lu2018iteration} and thus omitted.

\begin{lemma}\label{AL-Lipschitz}
Let $f_k$, $\mcQ$, $\chQ$, $L_k$ and $\tL_k$ be respectively defined in \eqref{fk}, \eqref{def-P}, \eqref{def-Q} and \eqref{conic-L}. Then $f_k$ is convex and continuously differentiable on $\dom(P)$, and moreover, $\nabla f_k$ is Lipschitz continuous on $\mcQ$ and $\chQ$ with Lipschitz constants $L_k$ and $\tL_k$, respectively.
\end{lemma}

The next lemma establishes some properties of $(x^k,\lambda^k)$ and $(x^k_*,\lambda^k_*)$.

\begin{lemma} \label{AL-bound}
Let $(x^k_*,\lambda^k_*)$ be defined in \eqref{ppa-sub}. Then the following statements hold.
\begin{align}
&\|(x^k,\lambda^k)-(x_*^k,\lambda_*^k)\|^2+\|(x_*^k,\lambda_*^k)-(x^*,\lambda^*)\|^2\leq\|(x^k,\lambda^k)-(x^*,\lambda^*)\|^2 \quad \forall 0 \leq k\in \bbK-1, \label{ppa-bnd1} \\
& \|(x^k,\lambda^k)-(x^{k-1},\lambda^{k-1})\| \leq\|(x^0,\lambda^0)-(x^*,\lambda^*)\|+\sum_{i=0}^{k-1}\rho_i\eta_i \quad \forall 1 \leq k\in \bbK, \label{ppa-bnd2} \\
&\|(x^k,\lambda^k)-(x^*,\lambda^*)\|\leq\|(x^0,\lambda^0)-(x^*,\lambda^*)\|+\sum_{i=0}^{k-1}\rho_i\eta_i \quad \forall 1 \leq k\in \bbK. \label{ppa-bnd3} 
\end{align}
\end{lemma}

\begin{proof}
It is well-known (e.g., see \cite{Rock76a,lu2018iteration}) that Algorithm \ref{PPA} is an inexact proximal point algorithm (PPA) applied to the monotone inclusion problem $0\in \mcT_l(x,\lambda)$, where $l$ is the Lagrangian function of problem~\eqref{conic-p}, and $\mcT_l$ is a maximal monotone set-valued operator defined as
\[
\mcT_l: (x,\lambda) \rightarrow \{ (v,u)\in\Re^n\times\Re^m: (v,-u)\in\partial l(x,\lambda)\}, \quad \forall (x,\lambda) \in \Re^n \times \Re^m. 
\]
It then follows from \eqref{subprob-term}, \eqref{ppa-sub}, and \cite[Lemma 5]{lu2018iteration} that  
\beq \label{PPA-step}
(x_*^k,\lambda_*^k)=\mcJ_{\rho_k}(x^k,\lambda^k), \qquad
\|(x^{k+1},\lambda^{k+1}) - \mcJ_{\rho_k}(x^k,\lambda^k)\| \leq \rho_k\eta_k, \quad \forall k\in\bbK-1.
\eeq
where $\mcJ_{\rho_k}=(\mcI + \rho_k\mcT_l)^{-1}$. By the first relation in \eqref{PPA-step}, 
$0\in\mcT_l(x^*,\lambda^*)$, and the maximal monotonicity of $\mcT_l$, it follows from \cite[Proposition 1]{Rock76a} that \eqref{ppa-bnd1} holds. In addition, \eqref{ppa-bnd2} and 
\eqref{ppa-bnd3} follow from the second relation in \eqref{PPA-step} and \cite[Lemma 3]{lu2018iteration}.
\end{proof}

As a consequence of Lemma~\ref{AL-bound} and the definition of $r_0$ and $\theta$ in \eqref{def1}, one has that 
\begin{equation}\label{AL-ineq}
\|x^0-x_*^0\|\leq r_0,\;\; \|x^k-x^*\|\leq r_0+\theta,\;\;\|\lambda^k-\lambda^*\|\leq r_0+\theta,\;\;\|x^k-x_*^k\|\leq r_0+\theta, \;\;\|x^k-x^{k-1}\|\leq r_0+\theta \quad \forall 1 \leq k\in \bbK.
\end{equation}

\begin{lemma}
Let $\tr0$ and $\br^k$ be defined in \eqref{def-tP} and \eqref{rk}.  Then for all $0 \le k\in\bbK-1$, we have
\begin{equation}\label{l11-potential}
\br_k^2 \leq\alpha_0^2\tr0^2\rho_k/2.
\end{equation}
\end{lemma}

\begin{proof}

We first prove that \eqref{l11-potential} holds for $k=0$, that is, $\br_0^2 \leq\alpha_0^2\tr0^2\rho_0/2$. Indeed, let $l$ be the Lagrangian function of problem~\eqref{conic-p}. By \eqref{conic-AL}, \eqref{fk} and \eqref{ppa-sub}, one has 
\begin{align*}
F_0(x^0_*) &=\mcL(x^0_*,\lambda^0;\rho_0)+\frac{1}{2\rho_0} \|x^0_*-x^0\|^2  \ge \mcL(x^0_*,\lambda^0;\rho_0) = \max_{\lambda\in\rr^m} \left\{ l(x^0_*,\lambda) - \frac{1}{2\rho_0}\|\lambda - \lambda^0\|^2\right\} \\
& \geq l(x^*_0,\lambda^*) - \frac{1}{2\rho_0}\|\lambda^0 - \lambda^*\|^2 \geq F(x^*)- \frac{1}{2\rho_0}\|\lambda^0 - \lambda^*\|^2,
\end{align*}
where the second equality follows from \cite[Lemma 2]{lu2018iteration}. Also, we have
\[
F_0(x^0) =\mcL(x^0,\lambda^0;\rho_0) = F(x^0)+\frac{1}{2\rho_0}\left(\|\Pi_{\mcK^*}(\lambda^0+\rho_0 g(x^0))\|^2-\|\lambda^0\|^2\right).
\] 
It then follows from these, \eqref{def-tP}, \eqref{rk}, and \eqref{AL-ineq} that
\begin{align*}
\br_0^2 &\overset{\eqref{rk}}{=} F_0(x^0)-F_0(x_*^0)+\frac{1}{2}\rho_0\alpha^2_0\|x^0-x_*^0\|^2 \\
& \leq F(x^0)-F(x^*)+\frac{1}{2\rho_0}\left(\|\Pi_{\mcK^*}(\lambda^0+\rho_0 g(x^0))\|^2+\|\lambda^0 - \lambda^*\|^2-\|\lambda^0\|^2\right) +\frac{1}{2}\rho_0\alpha^2_0r_0^2 \overset{\eqref{def-tP}}{\leq}\alpha_0^2\tr0^2\rho_0/2.
\end{align*}


We next show that \eqref{l11-potential} holds for all $1 \leq k\in\bbK-1$. Indeed, 
observe that $\|\lambda^{k}\|=\dist(\lambda^{k-1}+\rho_{k-1} g(x^{k}),-\mcK)$
 and $\|\Pi_{\mcK^*}(\lambda^k+\rho_k g(x^k))\|=\dist(\lambda^k+\rho_k g(x^k),-\mcK)$.
 Using these, $\rho_k=\rho_0\zeta^k$, and \eqref{AL-ineq}, we have
\begin{align}
\|\Pi_{\mcK^*}(\lambda^k+\rho_k g(x^k))-\lambda^k\|\leq& \ \dist(\lambda^k+\rho_k g(x^k),-\mcK)+\|\lambda^k\|=\rho_k\dist\left(\frac{\lambda^k}{\rho_k}+g(x^k),-\mcK\right)+\|\lambda^k\|\notag\\
\leq& \ \rho_k\dist\left(\frac{\lambda^k}{\rho_k}-\frac{\lambda^{k-1}}{\rho_{k-1}},-\mcK\right)+\rho_k\dist\left(\frac{\lambda^{k-1}}{\rho_{k-1}}+g(x^k),-\mcK\right)+\|\lambda^k\|\notag\\
\leq& \ \rho_k\left\|\frac{\lambda^k}{\rho_k}-\frac{\lambda^{k-1}}{\rho_{k-1}}\right\|+\frac{\rho_k}{\rho_{k-1}}\dist\left(\lambda^{k-1}+\rho_{k-1}g(x^k),-\mcK\right)+\|\lambda^k\|\notag\\
=& \ \rho_k\left\|\frac{\lambda^k}{\rho_k}-\frac{\lambda^{k-1}}{\rho_{k-1}}\right\|+\left(\frac{\rho_k}{\rho_{k-1}}+1\right)\|\lambda^k\| \leq \frac{\rho_k}{\rho_{k-1}}\|\lambda^{k-1}\|+\left(\frac{\rho_k}{\rho_{k-1}}+2\right)\|\lambda^k\|\notag\\
\leq& \ 2(\zeta+1)(\|\lambda^*\|+r_0+\theta)\label{l11-pi}.
\end{align}

It follows from \eqref{fk} and \eqref{subprob-term} that there exists $P'(x^k)\in\partial P(x^k)$ such that 
\begin{equation} \label{subg-F1}
F'_{k-1}(x^k)=\nabla f(x^k)+\nabla g(x^k)\Pi_{\mcK^*}(\lambda^{k-1}+\rho_{k-1} g(x^k))+\rho_{k-1}^{-1}(x^k-x^{k-1})+P'(x^k) \in \partial F_{k-1}(x^k), \quad \|F'_{k-1}(x^k)\|\leq\eta_{k-1}.
\end{equation}
Also, we have
\begin{equation*}
\nabla f(x^k)+\nabla g(x^k)\Pi_{\mcK^*}(\lambda^{k}+\rho_{k} g(x^k))+P'(x^k)\in\partial F_k(x^k),
\end{equation*}
which together with \eqref{subg-F1} yields
\beq \label{subg-F2}
F'_{k-1}(x^k)-\rho_{k-1}^{-1}(x^k-x^{k-1}) +\nabla g(x^k)\left(\Pi_{\mcK^*}(\lambda^{k}+\rho_{k} g(x^k))-\Pi_{\mcK^*}(\lambda^{k-1}+\rho_{k-1} g(x^k))\right)\in \partial F_k(x^k).
\eeq
In addition, observe from \eqref{def1} and \eqref{AL-ineq} that $x^k\in\ctQ$. Also, note that $F_k$ is convex and $g$ is $\widetilde L_g$-Lipschitz continuous on $\ctQ$. By these, \eqref{l11-pi}, \eqref{subg-F1}, \eqref{subg-F2}, and the monotonicity of $\{\rho_k\}$ and $\{\eta_k\}$, one has
\begin{align*}
F_k(x^k)-F_k(x_*^k)\overset{\eqref{subg-F2}}{\leq}& \ \langle F'_{k-1}(x^k),x^k-x_*^k\rangle-\rho_{k-1}^{-1}\langle x^k-x^{k-1},x^k-x_*^k\rangle\\
&+\langle\nabla g(x^k)(\Pi_{\mcK^*}(\lambda^k+\rho_k g(x^k))-\Pi_{\mcK^*}(\lambda^{k-1}+\rho_{k-1}g(x^k))),x^k-x_*^k\rangle\\
\leq&\ \| F'_{k-1}(x^k)\|\|x^k-x_*^k\|+\rho_{k-1}^{-1}\|x^k-x^{k-1}\|\|x^k-x_*^k\|\\
&+\|\nabla g(x^k)\|\|\Pi_{\mcK^*}(\lambda^k+\rho_k g(x^k))-\Pi_{\mcK^*}(\lambda^{k-1}+\rho_{k-1}g(x^k))\|\|x^k-x^k_*\|\\
= & \ \| F'_{k-1}(x^k)\|\|x^k-x_*^k\|+\rho_{k-1}^{-1}\|x^k-x^{k-1}\|\|x^k-x_*^k\|\\
&+\|\nabla g(x^k)\| \|\Pi_{\mcK^*}(\lambda^k+\rho_k g(x^k))-\lambda^k\|\|x^k-x^k_*\|\\
\leq& \ \eta_0(r_0+\theta)+\rho_0^{-1}(r_0+\theta)^2+2\widetilde L_g(\zeta+1)(\|\lambda^*\|+r_0+\theta)(r_0+\theta),
\end{align*}
where the last inequality follows from \eqref{AL-ineq} and \eqref{l11-pi}. Then we have
\begin{align*}
&\frac{2\br_k^2}{\rho_k\alpha_0^2}=\frac{2}{\rho_k\alpha_0^2}\left(F_k(x^k)-F_k(x_*^k)+\rho_k\alpha_0^{2}\|x^k-x_*^k\|^2\right) \leq \frac{2}{\rho_0\alpha_0^2}\left(F_k(x^k)-F_k(x_*^k)\right)+2\|x^k-x_*^k\|^2\notag\\
\leq&\frac{2}{\rho_0\alpha_0^2}\left(\eta_0(r_0+\theta)+\rho_0^{-1}(r_0+\theta)^2+2\widetilde L_g(\zeta+1)(\|\lambda^*\|+r_0+\theta)(r_0+\theta)\right)+2(r_0+\theta)^2\notag\\
=&\frac{2(r_0+\theta)}{\rho_0\alpha_0^2}\left(\eta_0+\rho_0^{-1}(r_0+\theta)+2\widetilde L_g(\zeta+1)(\|\lambda^*\|+r_0+\theta)+\rho_0\alpha_0^2(r_0+\theta)\right). 
\end{align*}
By this relation and the definition of $\tr0$ in \eqref{def-tP}, one can see that \eqref{l11-potential} holds for all $1 \leq k\in\bbK-1$.
\end{proof}

\begin{lemma}\label{l11}
Let $f_k$,  $L_k$, $\tL_k$, $\cS_k$ and $\chS_k$ be respectively defined in \eqref{fk}, \eqref{conic-L}, \eqref{def-Sk} and \eqref{def-hSk}. Then for all $0 \leq k\in\bbK-1$, $\nabla f_k$ is Lipschitz continuous on $\cS_k$ and  $\chS_k$ with Lipschitz constants $L_k$ and $\tL_k$, respectively.
\end{lemma}

\begin{proof}
Let $\mcQ$ and $\chQ$ be defined in \eqref{def-P} and \eqref{def-Q}. We first show that 
$\cS_k\subseteq\mcQ$ and $\chS_k\subseteq\chQ$ for all $0\leq k\in\bbK-1$. To this end, fix any $0\leq k\in\bbK-1$. By \eqref{def-Sk}, \eqref{AL-ineq} and \eqref{l11-potential}, one has that for all $x\in\cS_k$, 
\begin{equation*}
\|x-x^*\|\leq\|x-x_*^k\|+\|x_*^k-x^*\|\overset{\eqref{def-Sk} }{\leq} \alpha_0^{-1}\sqrt{2\rho_k^{-1}}\br_k +\|x_*^k-x^*\| \leq \tr0+r_0+\theta,
\end{equation*}
where the last inequality follows from \eqref{AL-ineq} and \eqref{l11-potential}. This together with \eqref{def-P} implies that $\cS_k\subseteq\mcQ$. In addition, by  \eqref{def1},  \eqref{def2}, \eqref{conic-L} and $\rho_k \ge \rho_0$, one has
\[
\rho_k^{-1}L_k\overset{\eqref{conic-L} }{=}C+\rho_k^{-1}B+\rho_k^{-1}L_{\nabla g}\sum_{i=0}^{k-1}\rho_i\eta_i+\rho_k^{-2}\overset{\eqref{def1}}{\leq} C+\rho_0^{-1}B+\rho_0^{-1}L_{\nabla g}\theta+\rho_0^{-2}\overset{\eqref{def2}}{=}L.
\]
Using this, \eqref{def-hSk} and \eqref{l11-potential}, we obtain that for all $x\in\chS_k$, 
\begin{align*}
\|x-x^*\| \leq  \|x-x_*^k\|+\|x_*^k-x^*\|\overset{\eqref{def-hSk} }{\leq} \left(1+L_k\rho_k^{-1}\right)\alpha_0^{-1}\sqrt{2\rho_k^{-1}}\br_k + \|x_*^k-x^*\| \leq  (1+L)\tr0+r_0+\theta.
\end{align*}
which  along with \eqref{def-Q} implies that $\chS_k\subseteq\chQ$.

The conclusion of this lemma then follows from Lemma \ref{AL-Lipschitz} and the fact that 
$\cS_k\subseteq\mcQ$ and $\chS_k\subseteq\chQ$ for all $0\leq k\in\bbK-1$.
\end{proof}

\begin{lemma}\label{prop-sub}
Let $N_k$ denote the number of evaluations of $\nabla f$, $\nabla g$, proximal operator of $P$ and projection onto $\mcK^*$ performed by Algorithm~\ref{alg-acc-term} at the $k$th outer iteration of  Algorithm~\ref{PPA}. Then for all $0 \le k\in \bbK-1$, it holds that 
\begin{equation}\label{def-Nk}
N_k \leq \wC_1\left(M+1+\frac{\left(\log\frac{\rho_k^2\alpha_0^2\tr0^2\left(\sqrt{\max\{1,\tL\delta^{-1}\}}+\tL\right)^2}{\eta_k^2}\right)_+}{\sqrt{(\mu+\rho_k^{-1})\rho_k^{-1}\min\left\{1,\delta \tL^{-1}\right\}}}\right),
\end{equation}
where $M$, $\delta$, $\alpha_0$, $\{\rho_k\}$ and $\{\eta_k\}$ are the input parameters of Algorithm~\ref{PPA}, and $\tr0$, $\tL$ and $\wC_1$ are given in \eqref{def-tP}, \eqref{def3} and \eqref{C1}, respectively. 
\end{lemma}

\begin{proof}
By  \eqref{def1},  \eqref{def3}, \eqref{conic-L} and $\rho_k \ge \rho_0$, one has
\beq\label{AL-Lbound}
\rho_k^{-1}\tL_k\overset{\eqref{conic-L} }{=}\wC+\rho_k^{-1}\tB+\rho_k^{-1}\tL_{\nabla g}\sum_{i=0}^{k-1}\rho_i\eta_i+\rho_k^{-2}\overset{\eqref{def1}}{\leq} \wC+\rho_0^{-1}\tB+\rho_0^{-1}\tL_{\nabla g}\theta+\rho_0^{-2}\overset{\eqref{def3}}{=}\tL.
\eeq
Notice that at the $k$th outer iteration of Algorithm~\ref{PPA}, Algorithm~\ref{alg-acc-term} is called to find an $\eta_k$-residual solution $x^{k+1}$ of the problem $\min_x \left\{f_k(x)+P(x)\right\}$ with the inputs $\epsilon \leftarrow\eta_k$, $\gamma_0 \leftarrow \rho_k^{-1}$,   $\mu \leftarrow \mu+\rho_k^{-1}$ and $x^1=z^1 \leftarrow x^k$. Moreover, when applied to this problem, the proximal step \eqref{unc-prox} of Algorithm~\ref{alg-acc-term} requires one evaluation of $\nabla f$, $\nabla g$, proximal operator of $P$ and projection onto $\mcK^*$, respectively. In view of this, \eqref{rk}, \eqref{def-Sk}, \eqref{def-hSk}, Lemma \ref{l11} and Theorem \ref{res-complexity}, one can replace $(r_0,\gamma_0, \mu, \epsilon, L_\chS)$ in \eqref{def-NN} by $(\br_k, \rho_k^{-1}, \mu+\rho_k^{-1}, \eta_k, \tL_k)$ respectively and obtain that 
\begin{align*}
N_k &\leq (1+M^{-1})\left(M+\left\lceil\frac{2\log\frac{\eta_k}{\br_k\left(\sqrt{2\max\{\rho_k,\tL_k\delta^{-1}\}}+\sqrt{2\rho_k^{-1}}\tL_k\right)}}{\log\left(1-\sqrt{(\mu+\rho_k^{-1})\min\left\{\rho_k^{-1},\delta \tL_k^{-1}\right\}}\,\right)}\right\rceil_+\right) \left(1+\left\lceil\frac{\log(\rho_k^{-1}\tL_k)}{\log(1/\delta)}\right\rceil_+\right) \notag \\
&\leq (1+M^{-1})\left(M+1+\frac{\left(\log\frac{2\rho_k\br_k^2\left(\sqrt{\max\{1,\rho_k^{-1}\tL_k\delta^{-1}\}}+\rho_k^{-1}\tL_k\right)^2}{\eta_k^2}\right)_+}{-\log\left(1-\sqrt{(\mu+\rho_k^{-1})\rho_k^{-1}\min\left\{1,\delta \rho_k\tL_k^{-1}\right\}}\,\right)}\right) \left(1+\left\lceil\frac{\log(\rho_k^{-1}\tL_k)}{\log(1/\delta)}\right\rceil_+\right) \notag \\
&\leq (1+M^{-1})\left(M+1+\frac{\left(\log\frac{2\rho_k\br_k^2\left(\sqrt{\max\{1,\rho_k^{-1}\tL_k\delta^{-1}\}}+\rho_k^{-1}\tL_k\right)^2}{\eta_k^2}\right)_+}{\sqrt{(\mu+\rho_k^{-1})\rho_k^{-1}\min\left\{1,\delta \rho_k\tL_k^{-1}\right\}}}\right) \left(1+\left\lceil\frac{\log(\rho_k^{-1}\tL_k)}{\log(1/\delta)}\right\rceil_+\right),
\end{align*}
where the last inequality follows from the fact that $-\log(1-\xi) \ge \xi$ for any $\xi \in(0,1)$.   By  the above inequality, \eqref{l11-potential} and \eqref{AL-Lbound}, one can see that \eqref{def-Nk} holds.
\end{proof}

%
%

We are now ready to prove Theorem~\ref{thm:constr}.

\begin{proof}[\textbf{Proof of Theorem~\ref{thm:constr}}.]
(i) Let $K$ be defined in \eqref{K}. We first show that Algorithm~\ref{PPA} terminates after at most $K+1$ outer iterations. Indeed, suppose for contradiction that it runs for more than $K+1$ outer iterations. It then follows that 
\eqref{ppa-term} does not hold for $k=K$. On the other hand, by \eqref{def1}, \eqref{ppa-bnd2}, \eqref{K}, $\rho_K=\rho_0\zeta^K$ and  $\eta_K=\eta_0\sigma^K$, one has 
\[
\frac{1}{\rho_K}\|(x^{K+1},\lambda^{K+1})-(x^K,\lambda^K)\| \leq \frac{r_0+\theta}{\rho_0\zeta^K} \overset{\eqref{K}}{\leq} \frac{\varepsilon}{2}, \qquad \eta_K=\eta_0\sigma^K\overset{\eqref{K}}{\leq}\frac{\varepsilon}{2},
\]
and hence \eqref{ppa-term} holds for $k=K$, which leads to a contradiction. In addition, the output of Algorithm~\ref{PPA} is an $\varepsilon$-KKT solution of problems \eqref{conic-p} and \eqref{conic-d} due to \cite[Theorem 4]{lu2018iteration}.

(ii) Suppose that $\mu=0$, i.e., $f$ is convex but not strongly convex on $\dom(P)$. 
Let $K$ and $\widehat N$ be defined in \eqref{K} and \eqref{def-complexity}. Also, let $N_k$ denote the number of evaluations of $\nabla f$, $\nabla g$, proximal operator of $P$ and projection onto $\mcK^*$ performed by Algorithm~\ref{alg-acc-term} at the $k$th outer iteration of Algorithm~\ref{PPA}. In addition to these evaluations, one projection onto $\mcK^*$ is performed at step 3 of Algorithm~\ref{PPA} each iteration. By these and statement (i) of this theorem, one can observe that the total number of evaluations of $\nabla f$, $\nabla g$, proximal operator of $P$ and projection onto $\mcK^*$ performed in Algorithm~\ref{PPA} is no more than $\sum_{k=0}^{|\bbK|-2}(N_k+1)$. 
As a result, to prove statement (ii) of this theorem, it suffices to show that $\sum_{k=0}^{|\bbK|-2} (N_k+1) \le \widehat N$. Indeed, in view of \eqref{C2}, \eqref{K}, \eqref{def-Nk}, $|\bbK|-2\leq K$, $\mu=0$, $\rho_k=\rho_0\zeta^k$ and $\eta_k=\eta_0\sigma^k$, one has
\begin{align*}
\sum_{k=0}^{|\bbK|-2} (N_k+1) & \leq K+1+ \wC_1 \sum_{k=0}^{K} \left(M+1+\frac{\rho_k\left(\log\frac{\rho_k^2\alpha_0^2\tr0^2\left(\sqrt{\max\{1,\tL\delta^{-1}\}}+\tL\right)^2}{\eta_k^2}\right)_+}{\min\left\{1,\sqrt{\delta \tL^{-1}}\right\}}\right) \\
& = K+1+\wC_1\sum_{k=0}^{K}\left(M+1+\frac{\rho_0\zeta^k\left(2k\log\frac{\zeta}{\sigma}+\log\frac{\rho_0^2\alpha_0^2\tr0^2\left(\sqrt{\max\{1,\tL\delta^{-1}\}}+\tL\right)^2}{\eta_0^2}\right)_+}{\min\left\{1,\sqrt{\delta \tL^{-1}}\right\}}\right)\\
& \leq K+1+\wC_1\left((M+1)(K+1)+\frac{\rho_0\zeta^{K+1}\left(2K\log\frac{\zeta}{\sigma}+\log\frac{\rho_0^2\alpha_0^2\tr0^2\left(\sqrt{\max\{1,\tL\delta^{-1}\}}+\tL\right)^2}{\eta_0^2}\right)_+}{(\zeta-1)\min\left\{1,\sqrt{\delta \tL^{-1}}\right\}}\right) \le \widehat N,
\end{align*}
where the first inequality follows from \eqref{def-Nk} and $\mu=0$, the second inequality is due to $\sum^K_{k=0}\zeta^k \le \zeta^{K+1}/(\zeta-1)$  and $\sum^K_{k=0}k \zeta^k \le K\zeta^{K+1}/(\zeta-1)$, and the last equality follows from \eqref{C2}, \eqref{K} and \eqref{def-complexity}.

(iii) Suppose that $\mu>0$, namely, $f$ is strongly convex on $\dom(P)$. Similar to the proof of statement (ii) of this theorem, it suffices to show that $\sum_{k=0}^{|\bbK|-2}(N_k+1) \le \check N$. Indeed, in view of \eqref{C2}, \eqref{K}, \eqref{def-Nk}, $|\bbK|-2\leq K$, $\mu>0$, $\rho_k=\rho_0\zeta^k$ and $\eta_k=\eta_0\sigma^k$, one has
\begin{align*}
 \sum_{k=0}^{|\bbK|-2}(N_k+1) & \leq K+1+\wC_1 \sum_{k=0}^{K} \left(M+1+\frac{\sqrt{\frac{\rho_k}{\mu}}\left(\log\frac{\rho_k^2\alpha_0^2\tr0^2\left(\sqrt{\max\{1,\tL\delta^{-1}\}}+\tL\right)^2}{\eta_k^2}\right)_+}{\min\left\{1,\sqrt{\delta \tL^{-1}}\right\}}\right) \\
& = K+1+\wC_1\sum_{k=0}^{K}\left(M+1+\frac{\sqrt{\frac{\rho_0}{\mu}}\sqrt{\zeta}^k\left(2k\log\frac{\zeta}{\sigma}+\log\frac{\rho_0^2\alpha_0^2\tr0^2\left(\sqrt{\max\{1,\tL\delta^{-1}\}}+\tL\right)^2}{\eta_0^2}\right)_+}{\min\left\{1,\sqrt{\delta \tL^{-1}}\right\}}\right)\\
& \leq K+1+\wC_1\left((M+1)(K+1)+\frac{\sqrt{\frac{\rho_0}{\mu}}\sqrt{\zeta}^{K+1}\left(2K\log\frac{\zeta}{\sigma}+\log\frac{\rho_0^2\alpha_0^2\tr0^2\left(\sqrt{\max\{1,\tL\delta^{-1}\}}+\tL\right)^2}{\eta_0^2}\right)_+}{(\sqrt{\zeta}-1)\min\left\{1,\sqrt{\delta \tL^{-1}}\right\}}\right) \leq \check N,
\end{align*}
where the first inequality follows from \eqref{def-Nk} and $\mu>0$, the second inequality is due to $\sum^K_{k=0}\sqrt{\zeta}^k \le \sqrt{\zeta}^{K+1}/(\sqrt{\zeta}-1)$  and
$\sum^K_{k=0}k \sqrt{\zeta}^k \le K\sqrt{\zeta}^{K+1}/(\sqrt{\zeta}-1)$, and the last equality follows from \eqref{C2}, \eqref{K} and \eqref{strong-complexity}.
\end{proof}

\section{Concluding remarks}\label{sec:conclude}

The development and analysis of accelerated first-order methods in this paper are based on the assumption that the proximal subproblems associated with $P$ can be exactly solved. Nevertheless, it is not hard to modify them by using a suitable inexact solution of the proximal subproblems instead.

Recently, a class of problems in the form of \eqref{unc-prob} with $f$ being relatively smooth convex was considered in the literature (e.g., see \cite{bauschke2017descent,hanzely2021accelerated,lu2018relatively}). Interestingly, this class consists of some problems in which $\nabla f$ is not locally Lipschitz continuous on $\cl(\dom(P))$, for example, the problem with $P$ being the simplex and $f$ containing the entropy function and being relatively smooth to the entropy function. It shall however be mentioned that this class generally does not include the problems considered in this paper. For example, it does not contain problem \eqref{unc-prob} with $f$ being a convex high-degree  polynomial function and $P$ being the indicator function of the nonnegative orthant. Yet, this problem belongs to the class considered in this paper. As future research, it would be interesting to investigate whether the methods studied in this paper can be extended to relatively smooth convex optimization.


\end{document}